\documentclass[12pt,reqno]{amsart}

\usepackage[a4paper,margin=1in]{geometry}
\usepackage{amsmath,amssymb,amsthm,mathtools,bm, mathrsfs}
\usepackage{microtype}
\usepackage{enumitem}
\usepackage{booktabs}
\usepackage[hidelinks]{hyperref}
\usepackage[nameinlink,noabbrev]{cleveref}
\usepackage{bbm}

\newtheorem{theorem}{Theorem}[section]
\newtheorem{proposition}[theorem]{Proposition}
\newtheorem{lemma}[theorem]{Lemma}
\newtheorem{corollary}[theorem]{Corollary}
\newtheorem{remark}[theorem]{Remark}
\newtheorem{definition}[theorem]{Definition}
\newtheorem{example}[theorem]{Example}
\newtheorem*{ac}{Acknowledgement}

\newcommand{\bC}{\mathbb{C}}

\newcommand{\bR}{\mathbb{R}}
\newcommand{\cL}{\mathcal{L}}

\newcommand{\cH}{\mathcal{H}}
\newcommand{\cS}{\mathcal{S}}
\newcommand{\cG}{\mathcal{G}}
\newcommand{\cJ}{\mathcal{J}}
\newcommand{\bE}{\mathbb{E}}
\newcommand{\cI}{\mathcal{I}}
\newcommand{\sP}{\mathscr{P}}
\newcommand{\Tr}{\mathrm{Tr}}
\newcommand{\Id}{\mathrm{Id}}

\newcommand{\bfM}{\mathbf{M}}

\newcommand{\KMS}{\mathrm{KMS}}

\newcommand{\mix}{\mathrm{mix}}
\newcommand{\diam}{\mathrm{diam}}
\newcommand{\Ent}{\mathrm{Ent}}

\allowdisplaybreaks[2]
\renewcommand{\Re}{\operatorname{Re}}
\renewcommand{\Im}{\operatorname{Im}}

\title[Quantum Cheeger Inequalities]{Quantum Cheeger Inequalities for KMS-Symmetric Quantum Markov Semigroups}

\author{Jincheng Wan}
\address{Jincheng Wan, Tsinghua University, Beijing}
\email{wanjc23@mails.tsinghua.edu.cn }

\author{Jinsong Wu}
\address{Jinsong Wu, Beijing Institute of Mathematical Sciences and Applications, Beijing, 101408, China}
\email{wjs@bimsa.cn}

\date{}

\begin{document}
	
	\begin{abstract}
		In this paper, we establish a quantum Cheeger inequality for primitive KMS-symmetric quantum Markov semigroups in terms of projection conductance. 
		We discuss both projection conductance and classical conductance for graph-based KMS-symmetric quantum Markov semigroups. 
		We show that hypercontractivity and the logarithmic Sobolev inequality hold for primitive KMS-symmetric quantum Markov semigroups. 
		We also present applications of the quantum Cheeger inequality to logarithmic Sobolev inequalities, hypercontractivity, and complete modified logarithmic Sobolev inequalities.
	\end{abstract}
	
	\maketitle
	
	\section{Introduction}
	
	In 1970, Cheeger \cite{Cheeger70} obtained a lower bound on the second eigenvalue of the Laplacian on a compact Riemannian manifold in terms of the manifold’s isoperimetric constant.
	This result is known as the Cheeger inequality. 
	Alon \cite{Alon86} proved a Cheeger inequality for the Laplacian of a finite graph. 
	Lawler and Sokal \cite{LawSok88} extended Cheeger-type inequalities to general Markov chains and Markovian jump processes.
	
	Quantum Markov semigroups are noncommutative generalizations of classical Markov chains. 
	The quantum counterparts of reversibility for classical Markov chains are GNS symmetry and KMS symmetry. 
	In this paper, we focus on primitive KMS-symmetric quantum Markov semigroups. 
	The generators and Dirichlet forms of KMS-symmetric quantum Markov semigroups have been investigated by Kossakowski, Frigerio, Gorini, and Verri \cite{KFGV77}, Fagnola and Umanità \cite{FagUma10}, Vernooij and Wirth \cite{VerWir23}, Wirth \cite{Wir26}, among others. 
	A quantum analogue of the Cheeger inequality for quantum channels was obtained by Temme, Kastoryano, Ruskai, Wolf, and Verstraete \cite{TKR10}. 
	In this paper, we obtain a quantum analogue of the Cheeger inequality for primitive KMS-symmetric quantum Markov semigroups. 
	We also find applications to characterizing the mixing time and the diameter of the state space with respect to Connes’ distance.
        Quantum Cheeger inequality characterizes the bound of the spectral gap in terms of conductance.
    For primitive KMS-symmetric quantum Markov semigroups, spectral-gap
estimates yield quantitative bounds for logarithmic Sobolev inequalities
and hypercontractivity. Under the additional Poissonization assumption,
they also yield bounds for the complete modified logarithmic Sobolev
inequality.
	
	The hypercontractivity and logarithmic Sobolev inequality for primitive quantum Markov semigroups were first formalized by Olkiewicz and Zegarlinski \cite{OlkZeg99}, who also established their equivalence. 
	Kastoryano and Temme \cite{KasTem13} developed quantum logarithmic Sobolev inequalities and their applications to rapid mixing. 
	This was later investigated by Carbone and Martinelli \cite{CarMar15} and by Temme, Pastawski, and Kastoryano \cite{TPK14}.
	
	Beigi, Datta, and Rouzé \cite{BDR20} proved hypercontractivity and quantum reverse hypercontractivity for GNS-symmetric quantum Markov semigroups via the quantum Stroock--\allowbreak Varopoulos inequality. 
	Bardet and Rouzé \cite{BarRou22} developed weak decoherence-free logarithmic Sobolev inequalities and hypercontractivity in amalgamated $L^p$ spaces. Their extension of Gross' integration lemma relates these properties, with an additional weak constant in the reverse implication depending on the center of the decoherence-free algebra. They also showed that the strong versions, without a defect term, fail in this framework for semigroups that are neither primitive nor unitary. 
	In this paper, we present a direct finite-dimensional proof of hypercontractivity for primitive KMS-symmetric quantum Markov semigroups, without an additional $L^p$-regularity assumption or the quantum Stroock--\allowbreak Varopoulos inequality.
	
	Modified logarithmic Sobolev inequalities (MLSI) and complete modified logarithmic Sobolev inequalities (CMLSI) have been studied intensively in \cite{BGJ22,BGJ23,GaoRou22,CarMaa17,GJLL25,LWW26}. In this paper, we show that primitive KMS-symmetric quantum Markov semigroups with Poissonized generators satisfy CMLSI and obtain a bound in terms of the projection conductance.

	The paper is organized as follows.
	In Section 2, we recall fundamentals of KMS-symmetric quantum Markov semigroups.
	In Section 3, we prove the quantum Cheeger inequality by introducing projection conductance.
	In Section 4, we compare the projection conductance and graph conductance in the graph-based KMS-symmetric quantum Markov semigroups.
	In Section 5, we introduce the graph-based KMS-symmetric quantum Markov semigroups with nonzero Hamiltonian for further study.
	In Section 6, we show that the complete modified logarithmic Sobolev inequality holds for Poissonized KMS-symmetric quantum Markov semigroups.
	In Section 7, we show that hypercontractivity and logarithmic Sobolev inequalities hold for primitive KMS-symmetric quantum Markov semigroups.
	
	\begin{ac}
		The paper is based on the numerous communications with ChatGPT Pro 5.6 and 6.
		J. Wu was supported by grants from Beijing Institute of Mathematical Sciences and Applications.
		J.~Wu was supported by NSFC (Grant no. 12371124). 
	\end{ac}

	\section{Preliminaries}
    Throughout the paper, we assume that $n\ge2$.
    
	In this section, we shall briefly recall the notation for KMS-symmetric quantum Markov semigroups on matrix algebras.
	
	Let $M_n(\bC)$ denote the algebra of $n\times n$ complex matrices, equipped with the unnormalized trace $\Tr$.
	Let $\sP$ be the space of all projections in $M_n(\bC)$ and $\cS(M_n(\bC))$ the space of density matrices in $M_n(\bC)$.
	We also denote $\cS(M_n(\bC))$ by $\cS(M_n)$ for simplicity.
	The Hilbert-Schmidt inner product $\langle \cdot, \cdot\rangle_{\mathrm{HS}}$ on $M_n(\bC)$ is defined as $\langle X,Y\rangle_{\mathrm{HS}}=\Tr(X^*Y)$ for $X, Y\in M_n(\bC)$. In the following, $\langle\cdot,\cdot\rangle$ denotes the Hilbert--Schmidt
inner product $\langle\cdot,\cdot\rangle_{\mathrm{HS}}$ unless otherwise specified.
	Suppose that $\rho$ is a faithful state on $M_n(\bC)$ with density matrix $D_\rho$.
	The KMS inner product $\langle\cdot, \cdot\rangle_{\KMS, \rho}$ is defined as 
	\begin{align*}
		\langle X, Y\rangle_{\KMS,\rho}= \Tr(X^*D_\rho^{1/2} Y D_\rho^{1/2}), \quad X, Y\in M_n(\bC).
	\end{align*}
	Let $s_-$ and $s_+$ denote the smallest and largest eigenvalues of $D_\rho$, respectively.
	Then for any $X\in M_n(\bC)$, we have that 
	\begin{align*}
		s_+^{-1}\|X\|^2_{\KMS, \rho} \leq   \|X\|^2_{\mathrm{HS}} \leq s_-^{-1}\|X\|^2_{\KMS, \rho}.
	\end{align*}
	For a linear map $\cL$ on $M_n(\bC)$, its KMS-norm is given by
	$\displaystyle \|\cL\|_{\KMS,\rho}=\sup_{X\neq 0}\frac{\|\cL(X)\|_{\KMS, \rho}}{\|X\|_{\KMS,\rho}}$

	Suppose that $\{\Phi_t\}_{t\geq 0}$ is a quantum Markov semigroup on the matrix algebra $M_n(\bC)$ with unnormalized trace $\Tr$.
	Let $\cL$ be the Lindbladian of $\{\Phi_t\}_{t\geq 0}$ such that $\Phi_t=e^{-t\cL}$.
	The GKSL representation \cite{GKS76,Lin76} of the generator is as follows:
	\begin{align}\label{eq:gksl}
		\cL(X)=i[H, X] +\sum_{j=1}^m \frac{1}{2}\{V_j^*V_j, X\}-\sum_{j=1}^m V_j^*XV_j,
	\end{align}
	where $H\in M_n(\bC)$ is Hermitian and $V_j\in M_n(\bC)$.
	A quantum Markov semigroup $\{\Phi_t\}_{t\geq 0}$ is called
	primitive if it admits a faithful invariant state and
	$\ker(\cL)=\bC I$. Equivalently, it has a unique invariant state
	$\rho$, which is faithful and satisfies $\rho\Phi_t=\rho$
	for all $t\geq 0$.
	
	A quantum Markov semigroup $\{\Phi_t\}_{t\geq 0}$ is KMS-symmetric with respect to a faithful state $\rho$ if for all $t\geq 0$ and $X,Y\in M_n(\bC)$,
	\begin{align}
		\Tr(X^* D_{\rho}^{1/2}\Phi_t(Y)D_\rho^{1/2})=\Tr(\Phi_t(X)^* D_{\rho}^{1/2}YD_\rho^{1/2}),
	\end{align}
	where $D_\rho$ is the density matrix of $\rho$, i.e. $\rho(\cdot)=\Tr(D_\rho \cdot)$.
	This implies that $\{\Phi_t\}_{t\geq 0}$ is in equilibrium with respect to $\rho$.
	For a quantum Markov semigroup admitting a faithful invariant state,
	we denote by $\bE_\Phi$ its ergodic projection onto $\ker\cL$:
	\[
	\bE_\Phi(X)
	=\lim_{T\to\infty}\frac1T\int_0^T\Phi_t(X)\,dt,
	\qquad X\in M_n(\bC).
	\]
	In finite dimensions, this limit exists in norm.
	If the semigroup is KMS-symmetric, then
	\[
	\bE_\Phi=\lim_{t\to\infty}\Phi_t.
	\]
	By Theorem 2.5 in \cite{VerWir23}, there exist matrices $W_1,\ldots,W_m\in M_n(\bC)$ such that  
	\begin{align}\label{eq:derivation}
		\langle X, \cL(Y)\rangle_{\KMS, \rho} =\sum_{j=1}^m \left\langle [W_j, X], [W_j, Y] \right\rangle_{\KMS,\rho}.
	\end{align}
We may choose the indexed family $\{W_j\}_{j=1}^m$ to be invariant under taking adjoints.
Since the KMS norm is invariant under taking adjoints and $\cL$ is $*$-preserving and KMS-symmetric,
$\langle X^*,\cL(X^*)\rangle_{\KMS,\rho}=\langle X,\cL(X)\rangle_{\KMS,\rho}$.
Thus the replacement $\widetilde{W}_j=\frac{W_j}{\sqrt{2}},\; \widetilde{W}_{m+j}=\frac{W_j^*}{\sqrt{2}},\; 1 \le j \le m$ preserves the quadratic form in \eqref{eq:derivation}, and polarization preserves the full identity.
	The Fisher information $\cI_{\cL}$ of the semigroup is
	\begin{align*}
		\cI_{\cL}(D)=\Tr(\cL^*(D)(\log D-\log \bE_{\Phi}^*(D))), 
	\end{align*}
	where $D$ is a faithful density matrix in $M_n(\bC)$ and $\cL^*$ is the adjoint of $\cL$ with respect to $\Tr$.
	When $\{\Phi_t\}_{t\geq 0}$ is primitive, we have that $\cI_{\cL}(D)=\Tr(\cL^*(D)(\log D-\log D_\rho))$.
	The relative entropy $H(D\| D_\rho)$ is $\Tr(D(\log D-\log D_\rho))$.
	For any faithful density matrices $D_1,D_2$, the integral representation
	of quantum relative entropy gives
	\begin{equation}\label{eq:relative-entropy-integral}
		H(D_1\Vert D_2)
		=
		\int_0^1(1-t)\,
		\left\langle
		D_1-D_2,\,
		\mathbf{K}_{(1-t)D_2+tD_1}^{-1}(D_1-D_2)
		\right\rangle\,dt.
	\end{equation}
	The proof can be found in Lemma 2.2 of \cite{GaoRou22}. 
	Recall that for any positive definite matrix $D\in M_n(\bC)$, the logarithmic-mean operator $\mathbf{K}_D$ on $M_n(\bC)$ is defined as 
	\[\mathbf{K}_D(X)=\int_0^1 D^s X D^{1-s}\; ds,\quad X \in M_n(\bC).\]
	Lemma 2.1 in \cite{GaoRou22} yields
	\begin{equation}\label{eq: K monotone}
		0<D_1\leq cD_2
		\quad\Longrightarrow\quad
		\left<X,\mathbf{K}_{D_1}^{-1}X \right> \geq c^{-1} \left<X, \mathbf{K}_{D_2}^{-1} X \right>.
	\end{equation}

	Let \(\{\Phi_t\}_{t\geq 0}\) be a quantum Markov semigroup admitting a faithful
	invariant state, with ergodic projection \(\mathbb{E}_{\Phi}\). We say that
	\(\{\Phi_t\}\) satisfies a \emph{modified logarithmic Sobolev inequality}
	(MLSI) if there exists \(\beta>0\) such that, for every density matrix \(D\)
	and every \(t\geq0\),
	\begin{align}\label{eq:mlsi}
		H(\Phi_t^*(D) \| \mathbb{E}_{\Phi}^*(D))
		\leq e^{-2\beta t} H(D \| \mathbb{E}_{\Phi}^*(D)),
	\end{align}
	equivalently, for every faithful density matrix \(D\),
	\begin{align}\label{eq:mlsi2}
		\mathcal{I}_{\mathcal{L}}(D)
		\geq 2\beta H(D \| \mathbb{E}_{\Phi}^*(D)).
	\end{align}
	In the primitive case, inequalities \eqref{eq:mlsi} and \eqref{eq:mlsi2}
	reduce to
	\[
	H(\Phi_t^*(D)\|D_\rho)\leq e^{-2\beta t}H(D\|D_\rho),
	\qquad
	\mathcal{I}_{\mathcal{L}}(D)\geq 2\beta H(D\|D_\rho).
	\]
	We denote by \(\alpha_{\mathrm{MLSI}}\) the optimal constant in inequalities
	\eqref{eq:mlsi} and \eqref{eq:mlsi2}.
	
	For \(\beta\geq0\), the semigroup satisfies
	\(\operatorname{CMLSI}(\beta)\) if, for every \(m\in\mathbb{N}\) and every
	faithful density matrix \(D\) on \(M_n\otimes M_m\),
	\begin{equation}\label{eq:cmlsi}
		\mathcal{I}_{\mathcal{L}\otimes\operatorname{Id}_m}(D)
		\geq 2\beta H\bigl(D\|(\mathbb{E}_\Phi^* \otimes \operatorname{Id}_m )(D)\bigr).
	\end{equation}
	If the semigroup is primitive with invariant density matrix \(D_\rho\), then
	\(\mathbb{E}_\Phi^*(X)=\operatorname{Tr}(X)D_\rho\). Consequently,
	\[
	(\mathbb{E}_\Phi^*\otimes\operatorname{Id}_m)(D)=D_\rho\otimes D_R,
	\qquad D_R=(\operatorname{Tr}_{M_n}\otimes\Id_m)(D).
	\]
	We denote by \(\alpha_{\mathrm{CMLSI}}(\mathcal{L})\) the supremum of all
	\(\beta\geq0\) for which \eqref{eq:cmlsi} holds for every \(m\in\mathbb{N}\)
	and every faithful density matrix \(D\in M_n\otimes M_m\).

	\section{Quantum Cheeger Inequality}
	
	In this section, we obtain a quantum Cheeger inequality for primitive KMS-symmetric quantum Markov semigroups.
	
	\begin{definition}[Projection Dirichlet Conductance]
		Suppose that $\{\Phi_t\}_{t\geq 0}$ is a primitive KMS-symmetric quantum Markov semigroup.
		The projection Dirichlet conductance is defined to be
		\begin{align*}
			h_\rho(\cL):=\inf_{0, I\neq P \in \sP}\frac{\langle P, \cL(P)\rangle_{\KMS,\rho}}{\min\{\rho(P), 1-\rho(P)\}}
		\end{align*}
	\end{definition}
	By primitivity and compactness of the set of nonzero proper projections, we have $h_\rho(\cL)>0$.
	
	\begin{remark}
		In the tracial case, with $\cL=\Id-S$ and $S$ the symmetrized quantum channel considered in \cite{TKR10}, this definition reduces to the projection conductance used there. 
		In the present setting, the generator $\cL$ is naturally interpreted as a Laplacian in the quantum framework: 
		it is the infinitesimal generator of a primitive KMS-symmetric quantum Markov semigroup, and its associated Dirichlet form plays the role of the noncommutative analogue of the classical Dirichlet energy. 
		Accordingly, projection conductance may be viewed as a quantum generalization of the classical Cheeger/isoperimetric constant, capturing the bottleneck behaviour of the semigroup through projections.
	\end{remark}

	Recall that the spectral gap $\lambda_{\KMS}$ of the semigroup is 
	\begin{align*}
		\lambda_{\KMS}=\inf_{\substack{0\neq X=X^*,\\ \rho(X)=0}}\frac{\langle X, \cL(X)\rangle_{\KMS, \rho}}{\|X\|_{\KMS,\rho}^2}.
	\end{align*}

	\begin{proposition}\label{prop:lowerbd}
		Suppose that $\{\Phi_t\}_{t\geq 0}$ is a primitive KMS-symmetric quantum Markov semigroup.
		We have that
		\begin{align*}
			\frac{h_\rho(\cL)^2s_-^2}{2s_+^2\|\cL\|_{\KMS, \rho}} \leq \lambda_{\KMS}.
		\end{align*}
	\end{proposition}
	\begin{proof}
		Suppose that $\displaystyle X=\sum_{j=1}^n x_jF_j$ with $\rho(X)=0$, where $\{F_j\}_{j=1}^n$ is an orthogonal family of minimal projections in $M_n(\bC)$ and $x_j\in \bR$.  
		Let $y_j=\rho(F_j)$.
		Then $y_j >0$ and $\displaystyle \sum_{j=1}^n x_j y_j=0$.
		
		Let 
		\begin{align*}
			w_{jk} =s_-\sum_{\ell=1}^m (\|F_jW_\ell F_{k}\|^2+ \|F_k W_\ell F_{j}\|^2),  \quad j \neq k.
		\end{align*}
		We define a classical  Markov generator $A$ on $\ell^2_n$ reversible with respect to $\vec{y}=(y_1, \ldots, y_n)$ as follows
		\begin{align*}
			(A \vec{t})_{j}=\frac{1}{y_j}\sum_{k\neq j} w_{jk} (t_j-t_k),
		\end{align*}
		where $(A \vec{t})_{j}$ is the $j$-th entry of the vector $A\vec{t}$.
		
		Let  
		\[
		h_A :=
		\min_{\varnothing \neq S \subsetneq \{1,\ldots,n\}}
		\frac{ \sum_{j\in S, \, k\notin S} w_{jk} }{\min\{\rho(F_S), 1-\rho(F_S)\}},
		\qquad
		F_S := \sum_{j\in S} F_j.
		\]
		Note that $h_A$ is the classical conductance.
		Then by Equation \eqref{eq:derivation}, we have that 
		\begin{align*}
			h_\rho(\cL)\leq & \frac{\langle F_S, \cL(F_S)\rangle_{\KMS, \rho}}{\min\{\rho(F_S), 1-\rho(F_S)\}}
			= \frac{\sum_{j=1}^m \langle [W_j, F_S], [W_j, F_S] \rangle_{\KMS,\rho} }{\min\{\rho(F_S), 1-\rho(F_S)\}} \\
			\leq & \frac{s_+ s_-^{-1} \sum_{j\in S, \, k\notin S} w_{jk} }{\min\{\rho(F_S), 1-\rho(F_S)\}}.
		\end{align*}
		Taking the minimum over all nonempty proper subsets $S$ yields
		\begin{align}\label{eq:classical}
			h_A \geq \frac{s_- h_\rho(\cL)}{s_+}.
		\end{align}
		
		Let $\lambda_A$ be the spectral gap of the classical Markov chain.
		The classical Cheeger inequality states that 
		\begin{align}\label{eq: classical Cheeger ineq}
			\frac{h_A^2}{2d_{\max}^A}\leq  \lambda_A \leq 2h_A,
		\end{align}
		where $\displaystyle d_{\max}^A=\max_{1\leq j\leq n}\{d_j^A\}$ and $\displaystyle d_j^A=\frac{\sum_{k\neq j} w_{jk}}{y_j}$.
		We have that 
		\begin{align*}
			d_j^A=& \frac{s_- \sum_{\ell=1}^m \|[W_\ell,F_j ]\|_{HS}^2}{\rho(F_j)}
			\leq \frac{ \sum_{\ell=1}^m \|[W_\ell,F_j ]\|_{\KMS, \rho}^2}{\rho(F_j)} \\
			= & \frac{ \langle F_j, \cL(F_j)\rangle_{\KMS, \rho}}{\rho(F_j)}\leq \|\cL\|_{\KMS, \rho} \frac{\| F_j\|_{\KMS,\rho}^2}{\rho(F_j)} \leq \|\cL\|_{\KMS, \rho}.
		\end{align*}
		This implies that $d_{\max}^A \leq \|\cL\|_{\KMS, \rho}$.
		We have that 
		\begin{align}\label{eq:lowerbdcl}
			\lambda_A \geq \frac{s_-^2 h_\rho(\cL)^2}{2s_+^2\|\cL\|_{\KMS, \rho}}.
		\end{align}
		Now we have that 
		\begin{align*}
			\langle X, \cL(X)\rangle_{\KMS, \rho} 
			\geq & s_-\sum_{\ell=1}^m \left\|[W_\ell,X ]\right\|_{HS}^2
			=\sum_{j<k} w_{jk}(x_j-x_k)^2 \\
			\geq & \lambda_A \sum_{j=1}^n y_j x_j^2
			\geq \frac{s_-^2 h_\rho(\cL)^2}{2s_+^2\|\cL\|_{\KMS, \rho}} \Tr(D_\rho X^2)\quad \text{ Equation \eqref{eq:lowerbdcl}}\\
			\geq & \frac{s_-^2 h_\rho(\cL)^2}{2s_+^2\|\cL\|_{\KMS, \rho}} \|X\|_{\KMS, \rho}^2
		\end{align*}
		This proves the proposition.
	\end{proof}
	
	\begin{remark}
		It is worth noting that the derivation matrices $W_\ell$ do not affect the lower bound of the spectral gap $\lambda_{\KMS}$ in the proof of Proposition \ref{prop:lowerbd}.     
	\end{remark}

	In the following, we shall obtain an upper bound for the spectral gap of $\cL$ in terms of the projection conductance.
	\begin{proposition}\label{prop:upperbd}
		We have that 
		\begin{align}\label{eq:cheegerupper}
			\lambda_{\KMS} \leq \left(\sqrt{\frac{s_-}{s_+}}+ \sqrt{\frac{s_+}{s_-}}\right) h_{\rho}(\cL).
		\end{align}
		Moreover, the upper bound is sharp.
	\end{proposition}
	\begin{proof}
		Let $P$ be a nonzero proper projection in $M_n(\bC)$.
		Note that 
		\begin{align*}
			\langle P-\rho(P)I, \cL (P-\rho(P)I)  \rangle_{\KMS, \rho}
			=\langle P, \cL(P) \rangle_{\KMS, \rho},
		\end{align*}
		and
		\begin{align*}
			\|P-\rho(P)I\|_{\KMS, \rho}^2 \geq \frac{2\sqrt{s_-s_+}}{s_-+s_+} \Tr(D_\rho (P-\rho(P)I)^2)
			=\frac{2\sqrt{s_-s_+}}{s_-+s_+} \rho(P)(1-\rho(P)).
		\end{align*}
		We have that 
		\begin{align*}
			\lambda_{\KMS}(\cL) \leq &  \frac{\langle P-\rho(P)I, \cL (P-\rho(P)I)  \rangle_{\KMS,\rho}}{\|P-\rho(P)I\|_{\KMS, \rho}^2} \\
			\leq & \frac{s_-+s_+}{2\sqrt{s_-s_+}}\frac{\langle P, \cL(P) \rangle_{\KMS, \rho}}{\rho(P)(1-\rho(P))}. 
		\end{align*}
		Taking the infimum over all nonzero proper projections $P$, we obtain
		\begin{align*}
			\lambda_{\KMS}(\cL)\leq \frac{s_-+s_+}{\sqrt{s_-s_+}} h_\rho(\cL).
		\end{align*}
		
		We consider the matrix algebra $M_2(\bC)$.
		Let 
		\begin{align*}
			D_\rho =\begin{pmatrix}
				s_1 & 0 \\ 0 &  s_2
			\end{pmatrix}, \quad V =\begin{pmatrix}
				0 & \sqrt{s_1} \\ \sqrt{s_2} & 0
			\end{pmatrix},\quad P =\frac{1}{2}\begin{pmatrix}
				1 & 1 \\ 1 & 1
			\end{pmatrix}.
		\end{align*}
		We define the linear map $\cL$ on $M_2(\bC)$ as follows: for any $X\in M_2(\bC)$,
		\begin{align*}
			\cL(X)=\gamma\left(\frac{1}{2} \{V^*V, X\}- V^*XV\right),
		\end{align*}
		where $\gamma>0$, $s_1,s_2>0$, $s_1+s_2=1$, and $s_1\ne s_2$.
These assumptions ensure that the example is primitive.
		For the given projection $P$, we have that 
		\begin{align*}
			\langle P, \cL(P) \rangle_{\KMS,\rho} =\frac{1}{2} \gamma\left(\frac{1}{2}-\sqrt{s_1s_2}\right)\sqrt{s_1s_2},
		\end{align*}
		and $\displaystyle \|P-\rho(P)I\|_{\KMS,\rho}^2=\frac12\sqrt{s_1s_2}$.
		We have that
		\begin{align*}
			h_\rho(\cL)\leq \gamma \left(\frac{1}{2}-\sqrt{s_1s_2}\right)\sqrt{s_1s_2}.
		\end{align*}
		On the other hand, Equation \eqref{eq:cheegerupper} implies that 
		\begin{align*}
			h_{\rho}(\cL) 
			\geq & \frac{\sqrt{s_1s_2}}{s_1+s_2}\lambda_{\KMS}
			=\sqrt{s_1s_2}\lambda_{\KMS} \\
			= & \gamma\left(\frac{1}{2}-\sqrt{s_1s_2}\right)\sqrt{s_1s_2}.
		\end{align*}
		This shows that the upper bound is attained and completes the proof of the proposition.
	\end{proof}

	\begin{theorem}[Quantum Cheeger Inequality]\label{thm:cheeger}
		Suppose that $\{\Phi_t\}_{t\geq 0}$ is a primitive KMS-symmetric quantum Markov semigroup.
		We have that 
		\begin{align*}
			\frac{h_\rho(\cL)^2s_-^2}{2s_+^2\|\cL\|_{\KMS,\rho}} \leq \lambda_{\KMS} \leq \left(\sqrt{\frac{s_-}{s_+}}+ \sqrt{\frac{s_+}{s_-}}\right) h_{\rho}(\cL).
		\end{align*}
	\end{theorem}
	\begin{proof}
		It follows from Propositions \ref{prop:lowerbd} and \ref{prop:upperbd} directly.
	\end{proof}
	
	We define the worst-case trace-distance mixing time $t_{\mix}$ for the semigroup $\{\Phi_t\}_{t\geq 0}$ as follows:
	for $0<\varepsilon<1$,
	\begin{align*}
		t_{\mix}(\varepsilon) =\inf\left\{ t\geq 0: \sup_{D\in \cS(M_n)} \|\Phi_t^*(D)-D_\rho\|_1\leq 2\varepsilon\right\},
	\end{align*}
	where $\|X\|_1=\Tr(|X|)$ for $X\in M_n(\bC)$.
	
	\begin{theorem}\label{thm: Quantum Cheeger ineq}
		Suppose that $\{\Phi_t\}_{t\geq 0}$ is a primitive KMS-symmetric quantum Markov semigroup.
		For $\displaystyle 0< \varepsilon < \frac{1}{2}$, we have that \begin{align*}
			\frac{\sqrt{s_+s_-}}{s_-+s_+} \frac{1}{h_{\rho}(\cL)}\log \frac{1}{2\varepsilon} \leq t_{\mix}(\varepsilon)\leq \frac{2s_+^2\|\cL\|_{\KMS, \rho}}{s_-^2 h_\rho(\cL)^2}\log \frac{\sqrt{s_-^{-1}-1}}{2\varepsilon}.
		\end{align*}   
	\end{theorem}
	\begin{proof}
		We first prove that
		\begin{equation}\label{eq: t mix bound}
			\frac{1}{\lambda_{\mathrm{KMS}}}\log\frac{1}{2\varepsilon}
			\le t_{\mix}(\varepsilon)
			\le \frac{1}{\lambda_{\mathrm{KMS}}}
			\log\frac{\sqrt{s_-^{-1}-1}}{2\varepsilon}.
		\end{equation}
		Let $X$ be a Hermitian matrix in $M_n(\bC)$ such that
		\[
		\mathcal L(X)=\lambda_{\KMS}X,
		\qquad \rho(X)=0,
		\qquad \|X\|_\infty=1,
		\]
		and $P_0$ is a rank-one projection such that $|\Tr(XP_0)|=1$.
		Then 
		\begin{align*}
			\|\Phi_t^*(P_0)-D_\rho\|_1
			\geq & |\Tr\bigl((\Phi_t^*(P_0)-D_\rho)X\bigr)| \\
			= &|\Tr\bigl((P_0-D_\rho)\Phi_t(X)\bigr)| \\
			=& e^{-\lambda_{\KMS}t}|\Tr(P_0X)-\rho(X)|=e^{-\lambda_{\KMS}t}.
		\end{align*}
		Hence
		\[
		2\varepsilon \ge \sup_{D \in \cS(M_n)}
		\|\Phi_{t_\mix(\varepsilon)}^*(D)-D_\rho\|_1
		\ge e^{-\lambda_{\KMS}t_{\mix}(\varepsilon)}.
		\]
		It follows that
		\begin{equation*}
			t_{\mix}(\varepsilon) \ge
			\frac{1}{\lambda_{\mathrm{KMS}}}\log\frac{1}{2\varepsilon}.
		\end{equation*}
		
		We now prove the upper bound of \eqref{eq: t mix bound}. 
		For any density matrix $D$, we have that 
		\begin{equation}
			\begin{aligned}
				\| \Phi_t^*(D)-D_\rho\|_1 
				&= \| D_\rho^{1/2} \Phi_t (D_\rho^{-1/2} D D_\rho^{-1/2}-I) D_\rho^{1/2} \|_1 \\
				&\le  \|\Phi_t (D_\rho^{-1/2} D D_\rho^{-1/2}-I) \|_{\KMS,\rho}\\
				&\le e^{-\lambda_\KMS t} \|D_\rho^{-1/2} D D_\rho^{-1/2}-I \|_{\KMS,\rho},
			\end{aligned}
		\end{equation}
		Here the first equality follows from KMS symmetry,
\[
\Phi_t^*(D_\rho^{1/2}XD_\rho^{1/2})
=D_\rho^{1/2}\Phi_t(X)D_\rho^{1/2}.
\]
The first inequality follows from the Cauchy--Schwarz inequality, and the second follows from the spectral gap.
Moreover, since $\Tr(D^2)\leq1$, we have
\[
\begin{aligned}
\|D_\rho^{-1/2}DD_\rho^{-1/2}-I\|_{\KMS,\rho}^2
&=\Tr\!\left[(D_\rho^{-1/4}DD_\rho^{-1/4})^2\right]-1\\
&\leq s_-^{-1}\Tr(D^2)-1\leq s_-^{-1}-1.
\end{aligned}
\]
		By taking
		\[
		t_*=
		\frac{1}{\lambda_{\KMS}}
		\log\frac{\sqrt{s_-^{-1}-1}}{2\varepsilon},
		\]
		we obtain that
		\[
		\sup_{D\in \cS(M_n)}
		\|\Phi_{t_*}^*(D)-D_\rho\|_1
		\le2\varepsilon.
		\]
		Hence $t_{\mathrm{mix}}(\varepsilon)\leq t_*$, which proves the second inequality in \eqref{eq: t mix bound}.
		
		Finally, substituting the bounds on $\lambda_{\KMS}$ from the quantum Cheeger inequality (Theorem \ref{thm:cheeger}) into inequality \eqref{eq: t mix bound} completes the proof of the theorem.
	\end{proof}
	
	\begin{remark}
		For the slow-mixing certification, it suffices to check finitely many projections with a small ratio.
		Theorem \ref{thm: Quantum Cheeger ineq} gives a lower bound on the necessary preparation time.
	\end{remark}

	We define a distance on the state space $\cS(M_n(\bC))$ by
	\begin{align*}
		d_{\cL}(D_1, D_2)=\sup_{\substack{X=X^*, \\ \langle X, \cL(X)\rangle_{\KMS, \rho}\leq 1}} |\Tr((D_1-D_2)X)|,
	\end{align*}
	where $D_1, D_2$ are density matrices.
	The diameter of the state space with respect to the distance $d_{\cL}$ is denoted by $\diam_{d_{\cL}}\cS(M_n(\bC))$. 
	
	Note that for a classical resistance network, the square of this distance between point masses is effective resistance.
	\begin{theorem}\label{thm:nongeodis1}
		Suppose that $\{\Phi_t\}_{t \geq 0}$ is a primitive KMS-symmetric quantum Markov semigroup.
		We have that 
		\begin{align*}
			\sqrt{\frac{2}{h_\rho(\cL)}} \leq \diam_{d_{\cL}} \cS(M_n(\bC)) \leq \frac{2s_+}{s_-h_\rho(\cL) } \sqrt{\frac{\|\cL\|_{\KMS, \rho}}{s_-}}.
		\end{align*}
	\end{theorem}
	\begin{proof}
		Suppose that $P$ is a projection such that $P\neq 0, I$ and $D_1, D_2$ are density matrices supported on $P$ and $I-P$ respectively.
		Then $\Tr((D_1-D_2)P)=1$ and 
		\begin{align}
			d_{\cL}(D_1, D_2) \geq \frac{1}{\sqrt{\langle P, \cL(P)\rangle_{\KMS, \rho}}}.
		\end{align}
		By taking the projection $P$ minimizing the conductance, we have that
		\begin{align*}
			\langle P, \cL(P)\rangle_{\KMS, \rho} =h_{\rho}(\cL) \min\{\rho(P), 1-\rho(P)\}\leq \frac{1}{2} h_{\rho}(\cL).
		\end{align*}
		By taking the supremum over all distances, we obtain that 
		\begin{align}
			\sqrt{\frac{2}{h_\rho(\cL)}} \leq \diam_{d_{\cL}} \cS(M_n(\bC)).   
		\end{align}
		
		Let
		\begin{align*}
			D_{-}=D_\rho^{-1/2}(D_1-D_2) D_\rho^{-1/2}.
		\end{align*}
		We have that $\rho(D_-)=0$ and 
		\begin{align*}
			d_{\cL}(D_1,D_2)=\sqrt{\langle D_-, \cL^{-1}(D_-)\rangle_{\KMS, \rho}},
		\end{align*}
		where $\cL^{-1}$ acts on the Hilbert subspace $\cH_0=\{X\in M_n(\bC): \rho(X)=0\}$.
		Now we have that 
		\begin{align*}
			d_{\cL}(D_1, D_2)^2 \leq & \frac{1}{\lambda_{\KMS}} \|D_-\|_{\KMS, \rho}^2
			\leq \frac{1}{\lambda_{\KMS}} s_-^{-1}\|D_1-D_2\|_{HS}^2\\
			\leq & \frac{2s_+^2 \|\cL\|_{\KMS, \rho}}{h_{\rho}(\cL)^2s_-^2} s_-^{-1} \|D_1-D_2\|_1 \quad \text{Theorem \ref{thm:cheeger}} \\
			\leq& \frac{4s_+^2 \|\cL\|_{\KMS, \rho}}{h_{\rho}(\cL)^2s_-^3}.
		\end{align*}
		This shows that 
		\begin{align*}
			\diam_{d_{\cL}} \cS(M_n(\bC)) \leq \frac{2s_+}{s_-h_\rho(\cL) } \sqrt{\frac{\|\cL\|_{\KMS, \rho}}{s_-}}.    
		\end{align*}
		Combining the two inequalities, we see that the inequality is true.
	\end{proof}
	
	\begin{remark}
		Theorem \ref{thm:nongeodis1} shows that small projection conductance implies large distance between quantum sectors.    
	\end{remark}
	
	An alternative way to define the distance is the Connes' distance
	\begin{align*}
		d_{\Gamma_0}(D_1, D_2)=\sup_{\substack{X=X^*,\\ \|\Gamma_0(X)\|\leq 1}}|\Tr((D_1-D_2)X)|,
	\end{align*}
	where $\displaystyle \Gamma_0(X)=\sum_{j=1}^m [W_j, X]^*[W_j , X]$.
	By the adjoint pairing of the indexed family $\{W_j\}_{j=1}^m$, for every $X=X^*$ we have
\[
\sum_j[W_j,X]^*[W_j,X]=\sum_j[W_j,X][W_j,X]^*.
\]
For any matrix $A$, the arithmetic--geometric mean inequality in an eigenbasis of $D_\rho$ gives
\[
\|A\|_{\KMS,\rho}^2
\leq\frac12\Tr\!\left(D_\rho(A^*A+AA^*)\right).
\]
Applying this to $A=[W_j,X]$ and summing yields
\[
\langle X,\cL(X)\rangle_{\KMS,\rho}
\leq\Tr(D_\rho\Gamma_0(X))\leq\|\Gamma_0(X)\|.
\]
	This shows that $d_{\Gamma_0}(D_1, D_2)\leq d_{\cL}(D_1, D_2)$,  and
	\begin{align*}
		\diam_{d_{\Gamma_0}}\cS(M_n(\bC)) \leq \diam_{d_{\cL}}\cS(M_n(\bC))\leq \frac{2s_+}{s_-h_\rho(\cL) } \sqrt{\frac{\|\cL\|_{\KMS, \rho}}{s_-}}.
	\end{align*}
	
	\section{Graph-based KMS-Symmetric QMS}\label{sec:graph}

In this section, we establish graph-dependent Cheeger bounds
for graph-based KMS-symmetric quantum Markov semigroups.
    
	Suppose that $\cG=(V(\cG), E(\cG))$ with the vertex set $V(\cG)=\{1,2, \ldots, n\}$ is a finite undirected graph with a positive weight $\gamma_{jk}>0$ on the edge $e_{jk}$ in $E(\cG)$, where $\gamma_{jk}=\gamma_{kj}$. 
	
	Let $\{E_{jk}\}_{j,k=1}^n$ be the system of matrix units of $M_n(\bC)$.
	Let $\rho$ be a faithful state on $M_n(\bC)$ with density $\displaystyle D_\rho=\sum_{j=1}^n s_j E_{jj}$.
We set $\gamma_{jk}=0$ whenever $e_{jk}\notin E(\cG)$.
Throughout this section, $n\geq2$. We use the Rayleigh quotient on the
space of matrices with zero $\rho$-mean to define $\lambda_{\KMS}$;
thus $\lambda_{\KMS}=0$ whenever $\ker\cL\neq\bC I$.
The quantity $\lambda_{\cG}$ denotes the second eigenvalue of the
unweighted graph Laplacian, counting multiplicity, and is therefore
zero for a disconnected graph.
	
	Let $d_j$ be the weighted degree with respect to $\rho$ at vertex $j$, which is defined as follows:
	\begin{align*}
		d_j=\sum_{k: e_{jk}\in E(\cG)}\gamma_{jk} s_k, \quad d_{\max}=\max_j\{d_j\}.
	\end{align*}
	Let $\vec{s}=(s_1, \ldots, s_n)^{\mathsf{T}}$ and
	\begin{align}\label{eq:spectralgap}
		\varphi_{\cG}=\inf_{\vec{s}\perp \vec{t},\  \vec{t}\neq 0}\frac{\displaystyle \sum_{e_{jk}\in E(\cG)}\gamma_{jk}s_js_k |t_j-t_k|^2}{\displaystyle \sum_{j=1}^n s_j|t_j|^2},
	\end{align}
	where $\vec{t}=(t_1, \ldots, t_n)^{\mathsf{T}}$.
	If $E(\cG)\neq\varnothing$, we have 
	\begin{align*}
		\frac{(\min_{e_{jk}}\gamma_{jk})\min_{e_{jk}}s_js_k}{\max_j s_j}   \lambda_{\cG}\leq \varphi_{\cG}\leq 2d_{\max},
	\end{align*}
	For an edgeless graph, $\varphi_{\cG}=d_{\max}=\lambda_{\cG}=0$,
and the generator defined below is zero.

	For each edge $e_{jk}\in E(\cG)$, we set
	\begin{align}
		V_{jk}=\sqrt{s_j}E_{jk}+\sqrt{s_k} E_{kj}\in M_n(\bC).
	\end{align}
	The Lindbladian associated to the graph $\cG$ is given by
	\begin{align}\label{eq:mnlind1}
		\cL(X) =\sum_{e_{jk}\in E(\cG)} \gamma_{jk} \left(\frac{1}{2}\{V_{jk}^*V_{jk}, X\}- V_{jk}^* X V_{jk}\right), \quad X\in M_n(\bC).
	\end{align}
	By the GKSL representation of the generator of the quantum Markov semigroup \cite{Lin76}, we see that the corresponding semigroup is a quantum Markov semigroup, denoted by $\{\Phi_t\}_{t\geq 0}$.
	Moreover, we have that 
	\begin{align}
		\cL(E_{jj}) =& \sum_{e_{jk}\in E(\cG)} \gamma_{jk} (s_k E_{jj}-s_jE_{kk}), \label{eq:lindbase1}\\
		\cL(E_{jk})= & \frac{d_j+d_k}{2} E_{jk}-\mathbbm{1}_{e_{jk}\in E(\cG)} \gamma_{jk}\sqrt{s_js_k}E_{kj}, \quad j\neq k. \label{eq:lindbase2}
	\end{align}
	
	Let 
	\begin{align}\label{eq:coherence}
		\phi_{\cG}=\min_{j<k}\frac{2\sqrt{s_js_k}}{s_j+s_k} \left(\frac{d_j+d_k}{2}-\mathbbm{1}_{e_{jk}}\gamma_{jk}\sqrt{s_js_k} \right).
	\end{align}
	We also define the conductance of $\cG$ by
	\begin{align*}
		h_{\rho, \cG}=\min_{\emptyset \neq S\subsetneq V(\cG)}\frac{\displaystyle \sum_{j\in S, \, k\notin S}\gamma_{jk} s_js_k}{\displaystyle \min\left\{\sum_{j\in S} s_j, \sum_{j\notin S} s_j\right\}}.
	\end{align*}
	\begin{proposition}
		Suppose that $\{\Phi_t\}_{t\geq 0}$ is a primitive KMS-symmetric quantum Markov semigroup with generator \eqref{eq:mnlind1}. 
		We have that 
		\begin{align*}
			\frac{1}{2}\min\{\varphi_{\cG}, \phi_{\cG}\}   \leq h_\rho(\cL)\leq h_{\rho, \cG}.
		\end{align*}   
	\end{proposition}
	\begin{proof}
		The upper bound follows by taking the diagonal test projections
$\displaystyle P_S=\sum_{j\in S}E_{jj}$ for all nonempty proper subsets $S\subset V(\cG)$.
		
		Suppose that $P$ is a nonzero proper projection with $(j,k)$-entry $p_{jk}$.
		By a direct computation, we have that
		\begin{align*}
			\langle P, \cL(P)\rangle_{\KMS,\rho}
			=& \sum_{j<k} \gamma_{jk} s_js_k(p_{jj}-p_{kk})^2 \\
			& +2\sum_{j<k} \sqrt{s_js_k}\left(\frac{d_j+d_k}{2}-\mathbbm{1}_{e_{jk}}\gamma_{jk}\sqrt{s_js_k} \right)(\Re p_{jk})^2 \\
			& +2\sum_{j<k} \sqrt{s_js_k}\left(\frac{d_j+d_k}{2}+\mathbbm{1}_{e_{jk}}\gamma_{jk}\sqrt{s_js_k} \right)(\Im p_{jk})^2.
		\end{align*}
		By Equations \eqref{eq:spectralgap} and \eqref{eq:coherence}, we have that 
		\begin{align*}
			\langle P, \cL(P)\rangle_{\KMS,\rho} \geq &  \varphi_{\cG}\left(\sum_{j} s_jp_{jj}^2-\rho(P)^2\right)
			+ \phi_{\cG} \sum_{j<k} (s_j+s_k)|p_{jk}|^2\\
			\geq & \varphi_{\cG} \left(\sum_{j} s_jp_{jj}^2-\rho(P)^2\right) +\phi_{\cG} \sum_{j} s_jp_{jj}(1-p_{jj})\\
			\geq & \min\{\varphi_{\cG}, \phi_{\cG}\}\left( \left(\sum_{j} s_jp_{jj}^2-\rho(P)^2\right) + \sum_{j} s_jp_{jj}(1-p_{jj})\right)\\
			=& \min\{\varphi_{\cG}, \phi_{\cG}\} \rho(P)(1-\rho(P)).
		\end{align*}
		By the fact that $\displaystyle \rho(P)(1-\rho(P))\geq \frac{1}{2}\min\{\rho(P), 1-\rho(P)\}$, we see that the lower bound is true.
	\end{proof}
	
	\begin{remark}
		Suppose that $n=2$ and $s_1\neq s_2$.
		The classical conductance $h_{\rho, \cG}=\gamma\max\{s_1, s_2\}$.
		We have that the projection conductance $h_{\rho}(\cL)$ is 
		\begin{align*}
			h_{\rho}(\cL)=\gamma\sqrt{s_1s_2}\left(\frac{1}{2}-\sqrt{s_1s_2}\right),
		\end{align*}
		and the minimum is attained at the projection $\displaystyle \frac{1}{2}\begin{pmatrix}
			1 & 1 \\ 1 & 1
		\end{pmatrix}$, which does not commute with $D_\rho$.
		By taking $\displaystyle s_1=\frac{1}{3}, s_2=\frac{2}{3}$, we see that $\displaystyle h_{\rho,\cG}=\frac{2}{3}\gamma$ and $\displaystyle h_{\rho}(\cL)=\frac{3\sqrt{2}-4}{18}\gamma < h_{\rho, \cG}$.
	\end{remark}

	In the following, we shall prove a quantum Cheeger inequality for graph-based KMS-symmetric quantum Markov semigroups.
	\begin{theorem}\label{thm:graphcheeger}
Suppose that $\{\Phi_t\}_{t\geq 0}$ is the graph-based KMS-symmetric quantum Markov semigroup defined by \eqref{eq:mnlind1}.
For an edgeless graph, we set $h_{\rho,\cG}^2/(2d_{\max})=0$.
If $n \geq 3$, we have that 
 \begin{align*}
\min\left\{ \frac{h_{\rho, \cG}^2}{2d_{\max}}, \frac{s_-}{2s_+} h_{\rho, \cG}\right\}     \leq \lambda_{\KMS} \leq 2h_{\rho, \cG}.
 \end{align*}
If $n=2$ and $\gamma_{12}=\gamma_{21}=\gamma$, then $\displaystyle \lambda_{\KMS}=\frac{\gamma}{2}(\sqrt{s_1}-\sqrt{s_2})^2$. 
\end{theorem}
\begin{proof}
If $\cG$ is disconnected, both $h_{\rho,\cG}$ and $\lambda_{\KMS}$
are zero, so the bounds are immediate, with the stated convention
when $d_{\max}=0$. We may therefore assume that $\cG$ is connected.
Set
\[
a_{jk}=\frac{d_j+d_k}{2}-\gamma_{jk}\sqrt{s_js_k},
\qquad \phi'_{\cG}=\min_{j<k}a_{jk}.
\]
The diagonal and off-diagonal invariant subspaces of $\cL$ give
\[
\lambda_{\KMS}=\min\{\varphi_{\cG},\phi'_{\cG}\}.
\]
The restriction of $\cL$ to the diagonal algebra is a reversible
Markov generator with invariant distribution $\vec{s}$, conductance
$h_{\rho,\cG}$, and maximal exit rate $d_{\max}$. The classical
Cheeger inequality \eqref{eq: classical Cheeger ineq} gives
\[
\frac{h_{\rho,\cG}^2}{2d_{\max}}
\leq\varphi_{\cG}\leq2h_{\rho,\cG}.
\]
In particular, this proves the required upper bound.

For $n\geq3$, choose $j<k$ such that $a_{jk}=\phi'_{\cG}$ and put
$S=\{j,k\}$. Since $S$ is a nonempty proper subset,
\[
\begin{aligned}
h_{\rho,\cG}\min\{s_j+s_k,1-s_j-s_k\}
&\leq\sum_{\ell\notin S}(\gamma_{j\ell}s_j+\gamma_{k\ell}s_k)s_\ell\\
&\leq s_+\sum_{\ell\notin S}(\gamma_{j\ell}+\gamma_{k\ell})s_\ell\\
&\leq2s_+a_{jk}.
\end{aligned}
\]
The last inequality follows from
\[
2a_{jk}=\sum_{\ell\notin S}(\gamma_{j\ell}+\gamma_{k\ell})s_\ell
+\gamma_{jk}(\sqrt{s_j}-\sqrt{s_k})^2.
\]
Moreover, $\min\{s_j+s_k,1-s_j-s_k\}\geq s_-$, so
\[
\phi'_{\cG}\geq\frac{s_-}{2s_+}h_{\rho,\cG}.
\]
Combining the two lower bounds proves the assertion for $n\geq3$.

If $n=2$, then $\varphi_{\cG}=\gamma$ and
\[
\phi'_{\cG}=\frac{\gamma}{2}(\sqrt{s_1}-\sqrt{s_2})^2
\leq\gamma.
\]
Hence $\lambda_{\KMS}=\phi'_{\cG}$, including the case $s_1=s_2$,
when the gap is zero. This completes the proof.
\end{proof}

	\section{Graph-based KMS-Symmetric QMS with nonzero Hamiltonian}
	
	In this section, we introduce a family of KMS-symmetric quantum Markov semigroups associated with an undirected graph, allowing a nonzero Hamiltonian.
	
	Suppose that $\cG=(V(\cG), E(\cG))$ with the vertex set $V(\cG)=\{1,2, \ldots, n\}$ is a finite undirected graph with a positive weight $\gamma_{jk}>0$ on the edge $e_{jk}$ in $E(\cG)$, where $\gamma_{jk}=\gamma_{kj}$. 
	For $e_{jk}\in E(\cG)$, let 
	\begin{align*}
		V_{jk}=\sqrt{\gamma_{jk}}(\sqrt{s_j}E_{jk}+\sqrt{s_k} E_{kj}).
	\end{align*}
	For each $e\in E(\cG)$, we define
	\begin{align*}
		W_e=\sum_{e'\in E(\cG)} t_{ee'}V_{e'}, \quad \widetilde{V}=\sum_{e\in E(\cG)}W_e^*W_e, 
	\end{align*}
	where $t_{ee'}\in \bR$.
	Let $T$ be the matrix whose $(e, e')$ entry is
	\begin{align*}
		\sum_{e''\in  E(\cG)} t_{e'' e} t_{e'' e'}.
	\end{align*}
	Let $t_-$ and $t_+$ denote the smallest and largest eigenvalues of $T$, respectively.
	We shall assume that $T$ is positive definite.
	Then $t_{\pm}>0$.

	Now we define a Hermitian matrix $H$ such that its $(j,j)$ entry is $0$ and $(j,k)$-entry is 
	\begin{align*}
		\frac{i}{2}\frac{\sqrt{s_k}-\sqrt{s_j}}{\sqrt{s_k}+\sqrt{s_j}}\widetilde{V}_{jk}, \quad j \neq k,
	\end{align*}
	where $\widetilde{V}_{jk}$ is the $(j,k)$-entry of $\widetilde{V}$.
	Then the generator $\cL$ is defined as: for $X\in M_n(\bC)$,
	\begin{align}\label{eq:graphlindh}
		\cL(X)=\frac{1}{2}\{\widetilde{V}, X\} -\sum_{e\in E(\cG)}W_e^* X W_e -i[H, X].  
	\end{align}
	
	Let 
	\begin{align*}
		\cL_0(X)= \frac{1}{2}\left\{\sum_{e\in E(\cG)}V_e^*V_e,  X\right\} -\sum_{e\in E(\cG)}V_e^* X V_e.
	\end{align*}
	Note that $\cL_0$ is a KMS-symmetric generator with respect to $\rho$.

	\begin{lemma}\label{lem:graphhambase}
		Suppose that $\cG$ is connected and either $n\geq3$, or $n=2$ with $s_1\neq s_2$.
		The semigroup with generator \eqref{eq:graphlindh} is a primitive KMS-symmetric quantum Markov semigroup.   
		Moreover, $\cL-t_-\cL_0$ and $t_+\cL_0-\cL$ are KMS-symmetric generators.
	\end{lemma}
	\begin{proof}
For a real symmetric positive semidefinite matrix \(S=(S_{ee'})\)
indexed by \(E(\cG)\), choose a real factorization \(S=R^{\mathsf T}R\)
and set
\[
 U_a=\sum_e R_{ae}V_e,\qquad
 B_S=\sum_a U_a^*U_a=\sum_{e,e'}S_{ee'}V_e^*V_{e'},
\]
where $R_{ae}$ is the $(a, e)$-entry of $R$.
Let \(H_S\) have zero diagonal and entries
\[
 (H_S)_{jk}
 =\frac{i}{2}\frac{\sqrt{s_k}-\sqrt{s_j}}
                     {\sqrt{s_k}+\sqrt{s_j}}(B_S)_{jk},
 \qquad j\ne k,
\]
and define
\[
 \cL_S(X)=\frac12\{B_S,X\}
          -\sum_a U_a^*XU_a-i[H_S,X].
\]
This is a quantum Markov generator and \(\cL_S(I)=0\).
The real coefficients in \(U_a\) and the identities
\begin{equation}\label{eq:graphrelation1}
 V_eD_\rho^{1/2}=D_\rho^{1/2}V_e^*,\qquad
 H_SD_\rho^{1/2}+D_\rho^{1/2}H_S
 =\frac{i}{2}(B_SD_\rho^{1/2}-D_\rho^{1/2}B_S)
\end{equation}
give
\[
 \cL_S^*(D_\rho^{1/2}XD_\rho^{1/2})
 =D_\rho^{1/2}\cL_S(X)D_\rho^{1/2}.
\]
Thus \(\cL_S\) is KMS-symmetric with respect to \(\rho\), and
\(\cL_S^*(D_\rho)=0\). For \(S=T\), this construction gives \(\cL\).
It also shows that \(H=0\) if and only if
\(\widetilde V D_\rho^{1/2}=D_\rho^{1/2}\widetilde V\).

To prove primitivity, let \(\cL(X)=0\). Apply the faithful invariant
state \(\rho\) to
\[
 \cL(X^*)X+X^*\cL(X)-\cL(X^*X)
 =\sum_e[W_e,X]^*[W_e,X].
\]
It follows that \([W_e,X]=0\); applying the same argument to \(X^*\)
also gives \([W_e^*,X]=0\). Since \(T>0\), the \(W_e\)'s span the
same space as the \(V_e\)'s, so \(X\) commutes with every \(V_e,V_e^*\).
These matrices generate \(M_n(\bC)\) under the stated assumptions.
Indeed, if \(n\ge3\), two adjacent edges \(ij,jk\) with distinct
endpoints give a nonzero multiple of \(E_{ik}\) as \(V_{ij}V_{jk}\).
Together with its adjoint this gives \(E_{ii}\); multiplication by
incident edge matrices and their adjoints, followed by connectivity,
then gives all matrix units. If \(n=2\) and \(s_1\ne s_2\), the
matrices \(V_{12},V_{12}^*\) already span \(E_{12},E_{21}\).
Consequently \(X\in\bC I\), proving primitivity.

Finally, \(B_S\), \(H_S\), and \(\cL_S\) depend linearly on \(S\).
For the identity matrix on the edge index set, \(B_I\) is diagonal,
so \(H_I=0\) and \(\cL_I=\cL_0\). Since
\[
 T-t_-I\ge0,\qquad t_+I-T\ge0,
\]
the construction above gives
\[
 \cL-t_-\cL_0=\cL_{T-t_-I},
 \qquad
 t_+\cL_0-\cL=\cL_{t_+I-T}.
\]
Both differences are therefore KMS-symmetric quantum Markov
generators with the common invariant density matrix \(D_\rho\).
\end{proof}
	
	\begin{proposition}
		Assume the hypotheses of Lemma~\ref{lem:graphhambase}, with generator \eqref{eq:graphlindh}.
		We have that 
		\begin{enumerate}
			\item $t_-\lambda_{\KMS}(\cL_0) \leq \lambda_{\KMS}(\cL)\leq t_+ \lambda_{\KMS}(\cL_0)$.
			\item $t_-\alpha_{MLSI}(\cL_0) \leq \alpha_{MLSI}(\cL)\leq t_+ \alpha_{MLSI}(\cL_0)$.
			\item $t_-\alpha_{CMLSI}(\cL_0) \leq \alpha_{CMLSI}(\cL)\leq t_+ \alpha_{CMLSI}(\cL_0)$.
		\end{enumerate}
	\end{proposition}
	\begin{proof}
Let \(\mathcal{A}_-=\cL-t_-\cL_0\) and
\(\mathcal{A}_+=t_+\cL_0-\cL\). By Lemma \ref{lem:graphhambase}, these are
KMS-symmetric quantum Markov generators preserving \(D_\rho\).
Their KMS quadratic forms are nonnegative, and hence
\[
 t_-\langle X,\cL_0(X)\rangle_{\KMS,\rho}
 \le \langle X,\cL(X)\rangle_{\KMS,\rho}
 \le t_+\langle X,\cL_0(X)\rangle_{\KMS,\rho}.
\]
The Rayleigh quotient on \(\{X:\rho(X)=0\}\) gives the spectral-gap
comparison.

For the entropy comparisons, fix \(m\ge1\) and a faithful density
matrix \(D\in M_n\otimes M_m\), and write
\[
 D_R=(\Tr_{M_n}\otimes\Id_m)(D),\qquad
 D_\eta=D_\rho\otimes D_R.
\]
For a generator \(\mathcal{A}\) preserving \(D_\rho\), define its entropy
production relative to this common reference state by
\[
 F_{\mathcal{A}}(D)
 =\Tr\!\left(((\mathcal{A}^*\otimes\Id_m)(D))
                  (\log D-\log D_\eta)\right).
\]
The state \(\eta\) is invariant under
\(e^{-t\mathcal{A}^*}\otimes\Id_m\). Data processing therefore gives
\[
 F_{\mathcal{A}}(D)
 =-\left.\frac{d}{dt}
 H\!\left((e^{-t\mathcal{A}^*}\otimes\Id_m)(D)\,\middle\|\,D_\eta\right)
 \right|_{t=0}
 \ge0.
\]
Applying this to \(\mathcal{A}_-\) and \(\mathcal{A}_+\), and using linearity in the
generator, yields
\[
 t_-F_{\cL_0}(D)\le F_{\cL}(D)\le t_+F_{\cL_0}(D).
\]
Both \(\cL_0\) and \(\cL\) are primitive with invariant density
matrix \(D_\rho\), so
\(F_{\cL_0}(D)=\cI_{\cL_0\otimes\Id_m}(D)\) and
\(F_{\cL}(D)=\cI_{\cL\otimes\Id_m}(D)\), with the same relative
entropy \(H(D\|D_\eta)\). Taking infima of
\(\displaystyle \frac{F_{\mathcal{A}}(D)}{2H(D\|D_\eta)}\) over \(D\ne D_\eta\), first for \(m=1\)
and then over all \(m\ge1\), proves the MLSI and CMLSI comparisons.
\end{proof}
	
	\begin{example}\label{ex:graphham}
		Let 
		\begin{align*}
			D_\rho=\frac{1}{14}\begin{pmatrix}
				1 & 0& 0 \\ 0 & 4 & 0 \\ 0 & 0& 9
			\end{pmatrix}, \quad 
			V_1=\begin{pmatrix}
				0 & 1 & 0 \\ 2 & 0 & 0 \\ 0 & 0 & 0
			\end{pmatrix}, \quad 
			V_2=\begin{pmatrix}
				0 & 0 & 0 \\ 0 & 0 & 2 \\ 0 & 3 & 0
			\end{pmatrix}, 
			\quad 
			H=\theta\begin{pmatrix}
				0 & 0& i \\ 0 & 0 & 0 \\ -i & 0& 0
			\end{pmatrix},
		\end{align*}
		where $-1\leq \theta \leq 1$.
		Let
		\begin{align*}
			W_+=\sqrt{\frac{1+\theta}{2}}(V_1+V_2), \quad W_-=\sqrt{\frac{1-\theta}{2}}(V_1-V_2).
		\end{align*}
		The matrix $\widetilde{V}=\begin{pmatrix}
			4 & 0& 4\theta \\ 0 & 10 & 0 \\ 4\theta & 0& 4
		\end{pmatrix}$.
		Then 
		\begin{align*}
			\cL(X)=\frac{1}{2} \{\widetilde{V}, X\}-W_+^*XW_+ - W_-^*XW_--i[H, X]
		\end{align*}
		is a KMS-symmetric quantum Markov generator. Indeed,
\[
 W_+^*W_++W_-^*W_-=\widetilde V,
 \qquad
 T=\begin{pmatrix}1&\theta\\ \theta&1\end{pmatrix}\ge0,
\]
and the KMS identities in \eqref{eq:graphrelation1} hold for these
matrices. 
Viewed as a linear operator on $M_3(\mathbb C)$, $\mathcal L$ has
the following nine eigenvalues, counted with multiplicity:
\[
0,\quad 4,\quad
3\pm\sqrt{4+3\theta^2},\quad
11\pm\sqrt{4+27\theta^2},\quad
r_1(\theta),r_2(\theta),r_3(\theta),
\]
where $r_1(\theta)\leq r_2(\theta)\leq r_3(\theta)$ are the three roots,
counted with multiplicity, of the cubic polynomial
\[
C_\theta(x)
=x^3-22x^2+(128-36\theta^2)x-224+168\theta^2.
\]

For every \(|\theta|\le1\), zero is simple and all other eigenvalues
are positive. 
Let $a=3-\sqrt{4+3\theta^2}\in[3-\sqrt7,1]$.
Using $a^2-6a+5-3\theta^2=0$, we obtain
\[
C_\theta(a)
=-24-6a-(1-\theta^2)(120-33a)<0.
\]
For $0<|\theta|\leq1$, we also have
\[
C_\theta(4)=24\theta^2>0,
\qquad C_\theta(14)=-336\theta^2<0.
\]
Since $C_\theta(x)\to+\infty$ as $x\to+\infty$, its three roots
lie in $(a,4)$, $(4,14)$, and $(14,\infty)$, respectively.
For $\theta=0$, the factorization
$C_0(x)=(x-4)^2(x-14)$ gives the same lower bound.
Together with
\[
11-\sqrt{4+27\theta^2}\geq11-\sqrt{31}>1\geq a,
\]
this shows that zero is a simple eigenvalue and the smallest
nonzero eigenvalue is $a>0$. Thus the QMS is primitive, with
\[
\lambda_{\mathrm{KMS}}=3-\sqrt{4+3\theta^2}.
\]
For \(\theta\ne0\), the Hamiltonian is nonzero.
	\end{example}


	\section{Connection to MLSI and CMLSI}
	
	In \cite{GaoRou22}, GNS-symmetric quantum Markov semigroups were shown to satisfy the complete modified logarithmic Sobolev inequality. 
	However, Liu and the present authors \cite{LWW26} showed that KMS-symmetric quantum Markov semigroups do not, in general, satisfy the complete modified logarithmic Sobolev inequality. 
	In this section, we study KMS-symmetric quantum Markov semigroups with Poissonized generators and show that they satisfy the complete modified logarithmic Sobolev inequality (CMLSI). 
	A positive KMS spectral gap does not, by itself, guarantee CMLSI. Nevertheless, the quantum Cheeger inequality yields a CMLSI bound for semigroups with Poissonized generators.
	\begin{lemma}\label{lem:kmspoisson}
		Let $\Psi:M_n(\mathbb C)\to M_n(\mathbb C)$ be a unital completely positive map satisfying \(\Psi^*(D_\rho)=D_\rho\), and assume that \(\Psi\) is KMS-symmetric with respect to a faithful invariant state $\rho$.
		If the fixed-point space of $\Psi$ is $\bC I$, then, for every $\gamma>0$, the operator
		\begin{align}\label{eq:poissongenerator}
			\cL=\gamma(\Id-\Psi) 
		\end{align}
		generates a primitive KMS-symmetric quantum Markov semigroup.
		If $\Psi$ is not GNS-symmetric, then $\cL$ is not GNS-symmetric.
	\end{lemma}
	\begin{proof}
		Suppose that $\displaystyle \Psi(X)=\sum_{j=1}^m V_j^* X V_j$, where $V_j\in M_n(\bC)$.
		By the fact that $\Psi(I)=I$, we have that $\displaystyle \sum_{j=1}^m V_j^*V_j=I$.
		This shows that 
		\begin{align*}
			\cL(X)=\gamma \left(\sum_{j=1}^m \frac{1}{2}\{V_j^*V_j, X\} -\sum_{j=1}^m V_j^* X V_j \right),
		\end{align*}
		and $\cL$ is a Lindbladian.
		If $\Psi$ is KMS-symmetric with respect to $\rho$, then 
		\begin{align*}
			\cL^*(D_\rho^{1/2}XD_\rho^{1/2})
			=&\gamma D_\rho^{1/2}XD_\rho^{1/2}- \gamma \Psi^*(D_\rho^{1/2}XD_\rho^{1/2}) \\
			= & \gamma D_\rho^{1/2}XD_\rho^{1/2}- \gamma D_\rho^{1/2}\Psi(X)D_\rho^{1/2}\\
			=&  D_\rho^{1/2}\cL(X)D_\rho^{1/2}
		\end{align*}
		i.e. $\cL$ is a KMS-symmetric generator.
		Moreover, $\cL$ is GNS-symmetric if and only if $\Psi$ is GNS-symmetric.
		
		Suppose that $X\in \ker\cL$.
		Then $\Psi(X)=X$.
		By assumption, $\ker\cL=\bC I$, so $\cL$ is a primitive quantum Markov generator.
	\end{proof}
	
	\begin{corollary}
		Suppose that $\cJ$ is a primitive KMS-symmetric generator.
		For every $s>0$ and $\gamma>0$, the operator $\cL=\gamma(\Id-e^{-s\cJ})$ is a primitive KMS-symmetric generator.
	\end{corollary}
	\begin{proof}
		It follows from Lemma \ref{lem:kmspoisson}.
	\end{proof}

	\begin{lemma}\label{lem:entropybd}
		For every faithful density matrix $D$, we have 
		\begin{align*}
			H(D \| D_\rho) \leq \log(1+\|D_\rho^{-1/2} D D_\rho^{-1/2}-I\|_{\KMS, \rho}^2) \leq \|D_\rho^{-1/2} D D_\rho^{-1/2}-I\|_{\KMS, \rho}^2.
		\end{align*}
	\end{lemma}
	\begin{proof}
		For any $\alpha \in (0,1) \bigcup (1,\infty)$,  the order-$\alpha$
		R\'enyi divergence of positive densities $D_2>0$ and $D_1>0$ is defined by
		\begin{equation*}
			\widetilde{H}_\alpha(D_2 \|D_1)=\displaystyle
			\frac{1}{\alpha-1}
			\log\!\left(
			\frac{
				\operatorname{Tr}\!\left[
				\left(
				D_1^{\frac{1-\alpha}{2\alpha}}
				D_2
				D_1^{\frac{1-\alpha}{2\alpha}}
				\right)^\alpha
				\right]}
			{\operatorname{Tr}(D_2)}
			\right).
		\end{equation*}
		By Theorem~7 in \cite{MFOSM2013}, for fixed $D_1,D_2>0$, the map
		$\alpha\mapsto\widetilde{H}_\alpha(D_1\Vert D_2)$ is monotonically increasing. And Theorem~5 in \cite{MFOSM2013} yields $H(D_2 \|D_1 )=\lim\limits_{\alpha \rightarrow 1^+} \widetilde{H}_\alpha(D_2 \|D_1 )$.
		Therefore,  we obtain
		\begin{align*}
			H(D\|D_\rho)
			&\leq \widetilde{H}_2(D\|D_\rho) = \log \operatorname{Tr}\!\left[
			\left(D_\rho^{-1/4} D D_\rho^{-1/4}\right)^2
			\right] \\
			&= \log\!\left(
			1+\left\|D_\rho^{-1/2} D D_\rho^{-1/2}-I
			\right\|_{\mathrm{KMS},\rho}^{2}
			\right) \\
			&\leq \left\|D_\rho^{-1/2} D D_\rho^{-1/2}-I
			\right\|_{\mathrm{KMS},\rho}^{2}.
		\end{align*}
This completes the proof of the lemma.
	\end{proof}

	\begin{theorem}\label{thm:poissonmlsi}
		Suppose that $\{\Phi_t\}_{t\geq 0}$ is a primitive KMS-symmetric quantum Markov semigroup whose generator is given by \eqref{eq:poissongenerator}.
		We have that 
		\begin{align*}
			\alpha_{MLSI} \geq \frac{ s_-^2}{4\gamma s_+}\lambda_{\KMS}^2\geq \frac{h_\rho(\cL)^4 s_-^6}{16 \gamma s_+^5\|\cL\|_{\KMS,\rho}^2}.
		\end{align*}
	\end{theorem}
	\begin{proof}
		We have that 
		\begin{align*}
			\frac{1}{\gamma} \cI_{\cL}(D)
			=& \Tr((D-\Psi^*(D))(\log D-\log D_\rho)) \\
			=& H(D\| D_\rho)-H(\Psi^*(D)\|D_\rho)+H(\Psi^*(D)\| D) \\
			\geq & H(\Psi^*(D)\| D) \quad \text{ Data Processing Inequality}\\
			\geq & \frac{1}{2} \|\Psi^*(D)- D\|_1^2 \quad \text{Pinsker Inequality} \\
			\geq & \frac{1}{2} \|\Psi^*(D)- D\|_{HS}^2=\frac{1}{2\gamma^2}\|\cL^*(D)\|_{HS}^2. 
		\end{align*}
		Note that 
		\begin{align*}
			\|\cL^*(D)\|_{HS}^2
			=& \|D_\rho^{1/2}\cL(D_\rho^{-1/2} D D_\rho^{-1/2}) D_\rho^{1/2}\|_{HS}^2 \\
			\geq & s_-^2  \| \cL(D_\rho^{-1/2} D D_\rho^{-1/2}) \|_{HS}^2  \\
			\geq & \frac{s_-^2}{s_+}\| \cL(D_\rho^{-1/2} D D_\rho^{-1/2}) \|_{\KMS, \rho}^2\\
			\geq & \frac{s_-^2}{s_+} \lambda_{\KMS} \langle D_\rho^{-1/2} D D_\rho^{-1/2}, \cL(D_\rho^{-1/2} D D_\rho^{-1/2})\rangle_{\KMS, \rho}\\
			= & \frac{s_-^2}{s_+} \lambda_{\KMS} \langle D_\rho^{-1/2} D D_\rho^{-1/2}-I, \cL(D_\rho^{-1/2} D D_\rho^{-1/2}-I)\rangle_{\KMS, \rho}\\
			\geq & \frac{s_-^2}{s_+} \lambda_{\KMS}^2 \|D_\rho^{-1/2} D D_\rho^{-1/2}-I\|_{\KMS, \rho}^2 \\
			\geq & \frac{s_-^2}{s_+}\lambda_{\KMS}^2 H(D\|D_\rho) \quad \text{Lemma \ref{lem:entropybd}},
		\end{align*}
		where the fifth line follows from the fact that $\cL(I)=0$ and $\cL$ is symmetric with respect to the inner product $\langle \cdot, \cdot\rangle_{\KMS, \rho}$.
		Combining the two estimates, we have that 
		\begin{align*}
			\cI_{\cL}(D) \geq \frac{s_-^2}{2\gamma s_+}\lambda_{\KMS}^2 H(D\|D_\rho).
		\end{align*}
		Applying Theorem~\ref{thm:cheeger} completes the proof.
	\end{proof}

\begin{theorem}\label{thm:poissoncmlsi}
Suppose that $\{\Phi_t\}_{t\geq 0}$ is a primitive KMS-symmetric quantum Markov semigroup whose generator is given by \eqref{eq:poissongenerator}.
We have that 
\begin{align*}
    \alpha_{CMLSI} \geq \frac{ s_-^2}{4n \gamma s_+}\lambda_{\KMS}^2\geq \frac{h_\rho(\cL)^4 s_-^6}{16 n \gamma s_+^5\|\cL\|_{\KMS,\rho}^2}.
\end{align*}
\end{theorem}
\begin{proof}
For arbitrary $m\geq 1$ , let $0<D \in M_n\otimes M_m$
be a density matrix and set
$D_R:=(\operatorname{Tr}_{M_n}\otimes\Id_m)(D)$.
Arguing as in the proof of Theorem~\ref{thm:poissonmlsi}, we have
\begin{equation}\label{eq: Fisher information expansion}
\frac{1}{\gamma}\mathcal{I}_{\cL \otimes \Id_m}(D) \geq  H\!\left(
(\Psi^*\otimes\mathrm{Id}_m)(D)
\,\middle\|\,D
\right).
\end{equation}

Note that
$D\leq nI_n\otimes
    (\Tr_{M_n}\otimes\mathrm{Id}_m)(D)$.
Since $\Psi^*$ is trace-preserving, $D$ and
$(\Psi^*\otimes\mathrm{Id}_m)(D)$ have the same marginal $D_R$,
so both satisfy the preceding bound, i.e.
$D\leq nI_n\otimes D_R,\;
 (\Psi^*\otimes\mathrm{Id}_m)(D)\leq nI_n\otimes D_R$.
Then the integral representation of relative entropy \eqref{eq:relative-entropy-integral} and the relation \eqref{eq: K monotone} yield
\begin{equation}\label{eq: integral rep of realtive entropy}
\begin{aligned}
 &H\!\left((\Psi^*\otimes\mathrm{Id}_m)(D)\,\middle\|\,D\right)\\
 =&\int_0^1(1-t)\,
 \left<(\Psi^*\otimes\mathrm{Id}_m)(D)-D, \mathbf{K}_{(1-t)D+t(\Psi^*\otimes\mathrm{Id}_m)(D)}^{-1} \left((\Psi^*\otimes\mathrm{Id}_m)(D)-D \right)\right>\,dt\\
 \geq&\frac1{2n}\,
 \left<(\Psi^*\otimes\mathrm{Id}_m)(D)-D, \mathbf{K}_{I_n \otimes D_R}^{-1}\left((\Psi^*\otimes\mathrm{Id}_m)(D)-D \right) \right>\\
 =&\frac{1}{2n \gamma^2} \left<(\cL^* \otimes \Id_m)(D-D_\rho \otimes D_R),\mathbf{K}_{I_n \otimes D_R}^{-1}(\cL^* \otimes \Id_m)(D-D_\rho \otimes D_R) \right>.
\end{aligned}
\end{equation}

Note that for any $X \in M_n(\bC)$ with $\Tr(X)=0$, 
\begin{equation}\label{eq:L dual HS}
\begin{aligned}
    \|\cL^*(X) \|_{\mathrm{HS}}^2&=\|D_\rho^{1/2} \cL(D_\rho^{-1/2} X D_\rho^{-1/2}) D_\rho^{1/2}\|_{\mathrm{HS}}^2 \\
    &\geq s_- \|\cL(D_\rho^{-1/2} X D_\rho^{-1/2})\|_{\KMS,\rho}^2
    \\
    &\geq s_- \lambda_\KMS^2 \| D_\rho^{-1/2} X D_\rho^{-1/2}\|^2_{\KMS, \rho} \\
    &\ge \frac{s_-}{s_+} \lambda_\KMS^2\|X\|_{\mathrm{HS}}^2.
\end{aligned}
\end{equation}
Now choose eigenvectors of $D_R$ and write
$ D_R=\operatorname{diag}(r_1,\ldots,r_m),
 \;
\displaystyle X=\sum_{j,k}X_{jk}\otimes E_{jk}$,
then
\begin{equation}\label{eq: Q for DR}
\left<X,\mathbf{K}_{I_n \otimes D_R}^{-1}X \right>
 =\sum_{j,k}\int_0^\infty
 \frac{dt}{(r_j+t)(r_k+t)}\|X_{jk}\|_{\mathrm{HS}}^2.
\end{equation}
Write $D-D_\rho \otimes D_R=\sum\limits_{j,k} D_{jk} \otimes E_{jk}$; then $\Tr(D_{jk})=0$. 
Applying \eqref{eq:L dual HS}, \eqref{eq: Q for DR} and relation \eqref{eq: K monotone}, we obtain
\begin{equation*}
\begin{aligned}
&\left<(\mathcal L^*\otimes\mathrm{Id}_m)
 (D-D_\rho\otimes D_R), \mathbf{K}_{I_n\otimes D_R}^{-1}(\mathcal L^*\otimes\mathrm{Id}_m)
 (D-D_\rho\otimes D_R) \right>\\
 =&\sum_{j,k}\int_0^\infty
 \frac{dt}{(r_j+t)(r_k+t)} \| \cL^*(D_{jk}) \|^2_{\mathrm{HS}}\\
 \geq& 
 \frac{s_-}{s_+} \lambda_\KMS^2\sum_{j,k}\int_0^\infty
 \frac{dt}{(r_j+t)(r_k+t)}  \|D_{jk}\|^2_{\mathrm{HS}}\\
 =
 &\frac{s_-}{s_+}\lambda_\KMS^2
 \left<
 D-D_\rho\otimes D_R, \mathbf{K}_{I_n\otimes D_R}^{-1}
 (D-D_\rho\otimes D_R) \right>\\
 \geq& \frac{s_-^2}{s_+}\lambda_\KMS^2\int_0^1(1-t)\, \left<(D-D_\rho\otimes D_R),\mathbf{K}^{-1}_{(1-t)D_\rho\otimes D_R+tD}
 (D-D_\rho\otimes D_R) \right>\,dt \\
 =&\frac{s_-^2}{s_+}\lambda_\KMS^2 H(D\|D_\rho \otimes D_R).
 \end{aligned}
\end{equation*}
Thus
\begin{equation}\label{eq: QD bigger than relative entropy}
\begin{aligned}
    &\left<(\mathcal L^*\otimes\mathrm{Id}_m)
 (D-D_\rho\otimes D_R), \mathbf{K}_{I_n\otimes D_R}^{-1}(\mathcal L^*\otimes\mathrm{Id}_m)
 (D-D_\rho\otimes D_R) \right> \\
 \geq &\frac{s_-^2}{s_+}\lambda_\KMS^2 H(D\|D_\rho \otimes D_R).
\end{aligned}
\end{equation}

Finally, combining inequalities \eqref{eq: Fisher information expansion}, \eqref{eq: integral rep of realtive entropy}, \eqref{eq: QD bigger than relative entropy}, we obtain 
\begin{equation*}
    \mathcal{I}_{\cL\otimes \Id_m}(D) \geq 2 \frac{ s_-^2}{4n \gamma s_+}\lambda_{\KMS}^2\;H (D\| D_\rho \otimes D_R).
\end{equation*}
The theorem then follows from the quantum Cheeger inequality (Theorem \ref{thm:cheeger}).
\end{proof}

\begin{remark}
  In \cite{LWW26}, Liu and the present authors show that the CMLSI is not true for every primitive KMS-symmetric quantum Markov semigroup.  
  Theorem \ref{thm:poissoncmlsi} shows that the CMLSI is true for primitive KMS-symmetric quantum Markov semigroups with Poissonized generators.
\end{remark}
	\begin{proposition}
		Suppose that $\cG$ is connected, $n\geq3$, and all weighted degrees are equal to a common value $d>0$.
Let $\{\Phi_t\}_{t\geq0}$ be the graph-based KMS-symmetric quantum Markov semigroup of Section~\ref{sec:graph}, and set $\gamma=d$.
		Then 
		\begin{align*}
			\alpha_{MLSI} \geq & \frac{ s_-^2}{4\gamma s_+}\lambda_{\KMS}^2 
			\geq \min\left\{ \frac{s_- ^2h_{\rho, \cG}^4}{16\gamma s_+d_{\max}^2}, \frac{s_-^4}{16\gamma s_+^3} h_{\rho, \cG}^2\right\}, \\
			\alpha_{CMLSI} \geq &  \frac{ s_-^2}{4n\gamma s_+}\lambda_{\KMS}^2 
			\geq \min\left\{ \frac{s_- ^2h_{\rho, \cG}^4}{16n\gamma s_+d_{\max}^2}, \frac{s_-^4}{16n\gamma s_+^3} h_{\rho, \cG}^2\right\}. 
		\end{align*}
	\end{proposition}
	\begin{proof}
		The assumptions imply that the semigroup is primitive. Moreover,
\[
\sum_{e_{jk}\in E(\cG)}\gamma_{jk}V_{jk}^*V_{jk}=dI.
\]
Hence the completely positive map
\[
\Psi(X)=\frac1d\sum_{e_{jk}\in E(\cG)}\gamma_{jk}V_{jk}^*XV_{jk}
\]
is unital, and $\cL=d(\Id-\Psi)=\gamma(\Id-\Psi)$.
The estimates follow by substituting Theorem~\ref{thm:graphcheeger} into Theorems~\ref{thm:poissonmlsi} and~\ref{thm:poissoncmlsi}.   
	\end{proof}

	\section{Connection to LSI and Hypercontractivity}
	
	In this section, we first establish the logarithmic Sobolev inequality and hypercontractivity for primitive KMS-symmetric quantum Markov semigroups and then apply the quantum Cheeger inequality.
	
	First of all, we recall the $p$-norm for noncommutative $L^p$ spaces for $1\leq p\leq \infty$.
	Suppose that $\rho$ is a faithful state on $M_n(\bC)$ and $X\in M_n(\bC)$.
	The $p$-norm of $X$ with respect to $\rho$ is defined to be
    \begin{align*}
    \|X\|_{p,\rho}
    :=
    \begin{cases}
        \displaystyle
        \left[
            \Tr\left(
                \left|
                    D_\rho^{\frac{1}{2p}} X D_\rho^{\frac{1}{2p}}
                \right|^p
            \right)
        \right]^{1/p},
        & 1 \le p < \infty, \\[8pt]
        \|X\|_\infty,
        & p = \infty.
    \end{cases}
\end{align*}
	Note that $\|X\|_{2, \rho}=\|X\|_{\KMS, \rho}$.
	Applying the noncommutative martingale convexity inequality of Ricard and Xu \cite[Theorem~1]{RicXu16} to the $\rho$-preserving conditional expectation $X\mapsto\rho(X)I$, we obtain, for any $X\in M_n(\bC)$ and $2\leq p<\infty$, 
	\begin{align}\label{eq:mart}
		\|X\|_{p, \rho}^2 \leq |\rho(X)|^2+(p-1) \|X-\rho(X)I\|_{p, \rho}^2
	\end{align}
	For $2\leq p<\infty$ and any $X\in M_n(\bC)$, we have that
	\begin{align*}
		\|X\|_{p, \rho} = & \left\|D_\rho^{\frac{1}{2p}} X D_\rho^{\frac{1}{2p}}\right\|_p \\
		= & \left\|D_\rho^{\frac{1}{2p}-\frac{1}{4}} D_\rho^{\frac{1}{4}}X D_\rho^{\frac{1}{4}}D_\rho^{\frac{1}{2p}-\frac{1}{4}}\right\|_p \\
		\leq & \left\|D_\rho^{\frac{1}{2p}-\frac{1}{4}} \right\|_\infty \left\|D_\rho^{\frac{1}{4}}X D_\rho^{\frac{1}{4}}\right\|_{p}\left\|D_\rho^{\frac{1}{2p}-\frac{1}{4}}\right\|_\infty \\
		\leq & \left\|D_\rho^{\frac{1}{2p}-\frac{1}{4}} \right\|_\infty \left\|D_\rho^{\frac{1}{4}}X D_\rho^{\frac{1}{4}}\right\|_{2}\left\|D_\rho^{\frac{1}{2p}-\frac{1}{4}}\right\|_\infty\\
		=& s_-^{1/p-1/2}\|X\|_{2, \rho}, 
	\end{align*}
	i.e.
	\begin{align}\label{eq:normbd}
		\|X\|_{p, \rho}\leq s_-^{1/p-1/2} \|X\|_{2, \rho}.
	\end{align}

	\begin{proposition}\label{prop:hyper1}
		Suppose that $\{\Phi_t\}_{ t\geq 0}$ is a primitive KMS-symmetric quantum Markov semigroup.
		For $2\leq q<\infty$, any $X\in M_n(\bC)$, and $\displaystyle t \geq \frac{2+\log s_-^{-1}}{4\lambda_{\KMS}}\log (q-1)$, we have that 
		\begin{align}\label{eq:hyperq2}
			\|\Phi_t(X)\|_{q, \rho} \leq \|X\|_{2, \rho}
		\end{align}
	\end{proposition}
	\begin{proof}
		Suppose that $X_0=X-\rho(X)I$.
		We have that $\Phi_t(X)=\rho(X)I+\Phi_t(X_0)$, $\rho(\Phi_t(X_0))=0$ and
		\begin{align*}
			\|X\|_{2, \rho}^2 =|\rho(X)|^2 +\|X_0\|_{2, \rho}^2.
		\end{align*}
		By the fact that the semigroup is primitive and KMS-symmetric, we have that 
		\begin{align}
			\|\Phi_t(X_0)\|_{2, \rho} \leq e^{-\lambda_{\KMS} t} \|X_0\|_{2, \rho}.
		\end{align}
		By Equations \eqref{eq:mart} and \eqref{eq:normbd}, we have that 
		\begin{align*}
			\|\Phi_t(X)\|_{q, \rho}^2
			\leq &  |\rho(X)|^2+(q-1)\|\Phi_t(X_0)\|_{q, \rho}^2 \\
			\leq &  |\rho(X)|^2+(q-1)s_-^{-1+2/q}\|\Phi_t(X_0)\|_{2, \rho}^2 \\
			\leq & |\rho(X)|^2+(q-1) s_-^{-1+2/q} e^{-2\lambda_{\KMS} t} \|X_0\|_{2, \rho}^2.
		\end{align*}
		By assuming that 
		\begin{align}\label{eq:timebd}
			(q-1)s_-^{-1+2/q} e^{-2\lambda_{\KMS} t}\leq 1,
		\end{align}
		we obtain that 
		\begin{align*}
			\|\Phi_t(X)\|_{q, \rho}^2\leq   |\rho(X)|^2+ \|X_0\|_{2, \rho}^2=\|X\|_{2, \rho}^2.
		\end{align*}
		Condition \eqref{eq:timebd} is satisfied whenever 
		\begin{align*}
			t\geq \frac{1}{2\lambda_{\KMS}}\left( \log (q-1) +\left( 1-\frac{2}{q} \right)\log s_-^{-1}\right),
		\end{align*}
		so inequality \eqref{eq:hyperq2} holds.
		By noting that $\displaystyle \log (q-1)\geq 2\left( 1-\frac{2}{q} \right)$, we see that when $\displaystyle t \geq \frac{2+\log s_-^{-1}}{4\lambda_{\KMS}}\log (q-1)$, the inequality \eqref{eq:hyperq2} is true.
		This completes the proof of the proposition.
	\end{proof}

	\begin{theorem}
		Suppose that $\{\Phi_t\}_{t\geq 0}$ is a primitive KMS-symmetric quantum Markov semigroup.
		For any $X\in M_n(\bC)$, we have that 
		\begin{align}\label{eq:hyper}
			\|\Phi_t(X)\|_{q, \rho} \leq \|X\|_{p, \rho},   \quad 1< p\leq q<\infty,
		\end{align}
		whenever $\displaystyle t \geq \frac{2+\log s_-^{-1}}{2\lambda_{\KMS}}\log \frac{q-1}{p-1}$.
	\end{theorem}
	\begin{proof}
		Suppose that $2\leq p \leq q$.
		Let $\displaystyle r=\frac{2q}{p}$.
		We have that $r\geq 2$.
		By Proposition \ref{prop:hyper1}, we have that 
		\begin{align*}
			\|\Phi_t(X)\|_{\frac{2q}{p}, \rho} \leq \|X\|_{2, \rho}.
		\end{align*}
		By taking the interpolation with $\|\Phi_t(X)\|_{\infty} \leq \|X\|_{\infty}$, we have that 
		\begin{align*}
			\|\Phi_t(X)\|_{q, \rho} \leq \|X\|_{p, \rho}
		\end{align*}
		whenever $\displaystyle t\geq \frac{2+\log s_-^{-1}}{4\lambda_{\KMS}}\log \left(\frac{2q}{p}-1\right)$.
		By noting that $\displaystyle \frac{2q}{p}-1 \leq \left(\frac{q-1}{p-1}\right)^2$, we have that when $\displaystyle t \geq \frac{2+\log s_-^{-1}}{2\lambda_{\KMS}}\log \frac{q-1}{p-1}$, the inequality \eqref{eq:hyper} is true.

        For $1<p\le q\le2$, weighted $L_p$ duality and KMS symmetry give
\[
    \|\Phi_t\|_{p\to q,\rho}
    =\|\Phi_t\|_{q'\to p',\rho}.
\]
Since $2\le q'\le p'$ and
$(p'-1)/(q'-1)=(q-1)/(p-1)$, the preceding case applies.
		
		Suppose that $p\leq 2\leq q$.
		This follows from the combination of the case $p\leq 2$ and the case $2\leq q$.
		
		Finally, the hypercontractivity estimate \eqref{eq:hyper} holds for $1<p\leq q <\infty$.
	\end{proof}
	
	By using the hypercontractivity for primitive KMS-symmetric quantum Markov semigroups, we can obtain the logarithmic Sobolev inequalities.
	With the convention $0\log0=0$, define
	\begin{equation}\label{eq:entropy}
\begin{aligned}
\Ent_{p,\rho}(X)
&=\frac{1}{p}\Tr\left(
\left|D_\rho^{\frac{1}{2p}}XD_\rho^{\frac{1}{2p}}\right|^p
\left(\log\left|D_\rho^{\frac{1}{2p}}XD_\rho^{\frac{1}{2p}}\right|^p-\log D_\rho\right)
\right)\\
&\quad-\frac{1}{p}\|X\|_{p,\rho}^p\log\|X\|_{p,\rho}^p.
\end{aligned}
\end{equation}
	For $X\neq 0$, we have that 
	\begin{align*}
		\Ent_{p, \rho}(X)=\frac{\|X\|_{p, \rho}^p}{p} H \left( \left.\frac{|D_\rho^{\frac{1}{2p}}XD_\rho^{\frac{1}{2p}}|^p}{\|X\|_{p, \rho}^p} \right\| D_\rho \right).
	\end{align*}
	For $X>0$ and $1<p,q<\infty$, the weighted power map $I_{p,q}$ is
	\begin{align*}
		I_{p,q}(X) = D_{\rho}^{-\frac{1}{2p}}\left( D_{\rho}^{\frac{1}{2q}} X D_{\rho}^{\frac{1}{2q}}\right)^{q/p} D_{\rho}^{-\frac{1}{2p}}.
	\end{align*}
	For $1<p<\infty$, the $p$-Dirichlet form $\mathcal{E}_{p, \cL}$ with respect to $\cL$ is 
	\begin{align}\label{eq:pvariance}
		\mathcal{E}_{p, \cL}(X)=\frac{p}{2(p-1)}\left\langle I_{\frac{p}{p-1}, p}(X), \cL(X)\right\rangle_{\KMS,\rho}.
	\end{align}
	The logarithmic Sobolev inequality constant $\alpha_p(\cL)$ is defined to be
	\begin{align*}
		\alpha_p(\cL) =\inf_{0<X\notin \bC I} \frac{\mathcal{E}_{p, \cL}(X)}{\Ent_{p,\rho}(X)}.
	\end{align*}
	
	\begin{theorem}[Logarithmic Sobolev Inequality]\label{thm:lsi}
		Suppose that $\{\Phi_t\}_{t\geq 0}$ is a primitive KMS-symmetric quantum Markov semigroup.
		We have that for $1<p<\infty$, 
		\begin{align}
			\alpha_p(\cL)\geq  \frac{\lambda_{\KMS}}{\log s_-^{-1}+2}\geq \frac{h_\rho(\cL)^2s_-^2}{2s_+^2\|\cL\|_{\KMS, \rho}} \frac{1}{\log s_-^{-1}+2}.
		\end{align}
		Furthermore,
		\begin{align*}
			\alpha_2(\cL) \geq \frac{2\lambda_{\KMS}}{\log s_-^{-1}+2} \geq \frac{2 h_\rho(\cL)^2s_-^2}{2s_+^2\|\cL\|_{\KMS, \rho}} \frac{1}{\log s_-^{-1}+2}.
		\end{align*}
	\end{theorem}
	\begin{proof}
		Let $\displaystyle \beta=\frac{2\lambda_{\KMS}}{\log s_-^{-1}+2}$ and
		\begin{align*}
			q(t)=1+(p-1)e^{\beta t}.
		\end{align*}
		For $X>0$ and $1<p<\infty$, differentiating \eqref{eq:hyper} at $t=0$ yields
\[
\begin{aligned}
0&\geq\left.\frac{d}{dt}\log\|\Phi_t(X)\|_{q(t),\rho}\right|_{t=0}\\
&=\frac{p-1}{p\|X\|_{p,\rho}^p}
\left(\beta\Ent_{p,\rho}(X)-2\mathcal E_{p,\cL}(X)\right).
\end{aligned}
\]
Thus $\mathcal E_{p,\cL}(X)\geq(\beta/2)\Ent_{p,\rho}(X)$ and $\alpha_p(\cL)\geq\beta/2$.
For $p=2$, Proposition~\ref{prop:hyper1} allows the faster curve $q(t)=1+e^{2\beta t}$; the same calculation gives $\alpha_2(\cL)\geq\beta$.
The remaining bounds follow from Theorem~\ref{thm:cheeger}.
	\end{proof}

\begin{remark}
For a faithful density matrix $D$, set $X=D_\rho^{-1/2}DD_\rho^{-1/2}$. As $p\downarrow1$,
\[
    \mathcal E_{p,\mathcal L}(X)\longrightarrow
    \frac12\mathcal I_{\mathcal L}(D),
    \qquad
    \operatorname{Ent}_{p,\rho}(X)\longrightarrow H(D\|D_\rho).
\]
Thus Theorem \ref{thm:lsi} yields
\[
    \alpha_{\mathrm{MLSI}}(\mathcal L)
    \ge \frac{\lambda_{\mathrm{KMS}}}{2+\log s_-^{-1}}.
\]
This provides an alternative lower bound of the MLSI constant for primitive KMS-symmetric quantum Markov semigroup. This gives an alternative quantitative proof of the MLSI established
in \cite{LWW26} for primitive KMS-symmetric quantum Markov semigroups.
\end{remark}
	
	\begin{remark}
		For the primitive KMS-symmetric quantum Markov semigroup described in Example~\ref{ex:graphham}, we have that $\lambda_{\KMS}=3-\sqrt{4+3\theta^2}\geq3-\sqrt{7}$ and
		\begin{align*}
			\alpha_p(\cL) \geq \frac{3-\sqrt{4+3\theta^2}}{2+\log14}.
		\end{align*}
	\end{remark}
	
	\begin{proposition}
		Suppose that $\{\Phi_t\}_{t\geq0}$ is the graph-based KMS-symmetric quantum Markov semigroup of Section~\ref{sec:graph}, associated with a connected graph $\cG$ on $n\geq3$ vertices.
		Then for $1<p<\infty$, the logarithmic Sobolev constant satisfies
		\begin{align*}
			\alpha_p(\cL) \geq \frac{1}{\log s_-^{-1}+2}\min\left\{ \frac{h_{\rho, \cG}^2}{2d_{\max}}, \frac{s_-}{2s_+} h_{\rho, \cG}\right\}. 
		\end{align*}
	\end{proposition}
	\begin{proof}
		This follows from Theorems~\ref{thm:graphcheeger} and~\ref{thm:lsi}.
	\end{proof}
	
	\begin{remark}\label{rem:lsifailure}
		The primitivity of quantum Markov semigroups does not guarantee the logarithmic Sobolev inequalities. 
		We shall test the example described in \cite[Section 4]{LWW26}.
		Let
		\begin{align*}
			P=\begin{pmatrix}
				0 & 1 \\ 1 & 0
			\end{pmatrix}, 
			\quad 
			Q=\begin{pmatrix}
				0 & -i \\ i & 0
			\end{pmatrix},
			\quad
			W=\begin{pmatrix}
				1 & 0 \\ 0 & -1
			\end{pmatrix}
		\end{align*}
		be the Pauli matrices in $M_2(\bC)$ and $\displaystyle \rho=\frac{1}{2}\Tr$.
		For any $X\in M_2(\bC)$, we define
		\begin{align}\label{eq:m2lind0}
			\cL(X)=i [P,X]+X-WXW.
		\end{align}
		Let $\Phi_t=e^{-t\cL}$.
		Then it is a primitive quantum Markov semigroup but not KMS-symmetric as shown in \cite{LWW26}.
		Suppose that $0<\varepsilon<1$ and $\displaystyle D_{\varepsilon}=\frac{1}{2}(I+\varepsilon W)$.
		We have that $\cL(D_\varepsilon)=\varepsilon Q$.
		By a direct computation, we have that for $1<p < \infty$,
		\begin{align*}
			\Tr(D_\varepsilon^{p-1}\cL(D_\varepsilon))=0.
		\end{align*}
		This implies that $\mathcal{E}_{p, \cL}(D_\varepsilon)=0$ for $1<p<\infty$.
		On the other hand, we have that
		\begin{align*}
			\Ent_{p,\rho}(D_\varepsilon)
&=\frac{1}{2^{p+1}p}\Biggl[
(1+\varepsilon)^p\log\frac{2(1+\varepsilon)^p}{(1+\varepsilon)^p+(1-\varepsilon)^p}\\
&\qquad\qquad+(1-\varepsilon)^p\log\frac{2(1-\varepsilon)^p}{(1+\varepsilon)^p+(1-\varepsilon)^p}
\Biggr]>0.
		\end{align*}
		Since $\Phi_t$ is contractive on $L_p(\rho)$,
$\mathcal E_{p,\mathcal L}$ is nonnegative.
Hence $\alpha_p(\mathcal L)=0$ for $1<p< \infty$.
		The failure of the modified logarithmic Sobolev inequality for the same example is established in \cite{LWW26}.
	\end{remark}

	\begin{remark}
		The primitivity of quantum Markov semigroups does not guarantee the hypercontractivity either.   
		We still take the example in Remark \ref{rem:lsifailure}.
		For any $X=aI +xP+yQ+zW\in M_2(\bC)$ with $a, x, y, z\in \bC$, we have that
		\begin{align*}
			\Phi_t(X)=aI+ (P, Q, W) \bfM (x, y, z)^{\mathsf{T}},
		\end{align*}
		where $\bfM$ is a $3\times 3$ real matrix as follows:
		\begin{align*}
			\bfM= \begin{pmatrix}
				e^{-2t} & 0 & 0 \\
				0 & e^{-t} \left(\cos(\sqrt{3} t)-\frac{\sqrt{3}}{3}\sin(\sqrt{3}t)\right) &  \frac{-2\sqrt{3}}{3}e^{-t}\sin(\sqrt{3}t) \\
				0& \frac{2\sqrt{3}}{3}e^{-t}\sin(\sqrt{3}t) & e^{-t} \left(\cos(\sqrt{3} t)+\frac{\sqrt{3}}{3}\sin(\sqrt{3}t)\right)
			\end{pmatrix} 
		\end{align*}
		The largest singular value of $\bfM$ is 
		\begin{align*}
			e^{-t}\exp\left(\mathrm{arsinh}\left(\frac{|\sin(\sqrt{3}t)|}{\sqrt{3}}\right)\right)
		\end{align*}
		where $\mathrm{arsinh}(s)=\log\left(s+\sqrt{1+s^2}\right), s\in \bR$.
		Let 
		\begin{align*}
			\alpha(t)=t-\mathrm{arsinh}\left(\frac{|\sin(\sqrt{3}t)|}{\sqrt{3}}\right).
		\end{align*}
		For $1<p\leq q<\infty$, the condition $\|\Phi_t\|_{p\to q}=1$ necessarily implies
\[
q\leq 1+(p-1)e^{2\alpha(t)},
\]
where
		\begin{align*}
			\|\Phi_t\|_{p\to q} =\sup_{X\neq 0}\frac{\|\Phi_t(X)\|_{q, \rho}}{\|X\|_{p,\rho}}.
		\end{align*}
		Indeed, for any nonzero traceless self-adjoint $A$, expansion at the identity gives
\[
\|I+\varepsilon A\|_{p,\rho}
=1+\frac{p-1}{2}\varepsilon^2\rho(A^2)+O(\varepsilon^3).
\]
Applying the contraction inequality to $I+\varepsilon A$ and taking a direction attaining the largest singular value of $\bfM$ yields
$(q-1)e^{-2\alpha(t)}\leq p-1$.
Since $\displaystyle \alpha(t)=\frac23t^3+O(t^5)$ as $t\downarrow0$, we have $\displaystyle \frac{\alpha(t)}{t}\to0$. Consequently, no finite constant $\beta>0$ can guarantee
$\|\Phi_t\|_{p\to q}=1$ for all $1<p\leq q<\infty$ whenever
$\displaystyle t\geq\beta\log\frac{q-1}{p-1}$.
This rules out a uniform positive-rate exponential hypercontractivity curve starting at $t=0$; it does not rule out $p$-to-$q$ contraction at an individual positive time.
	\end{remark}
	
	
		
	
	\bibliographystyle{abbrv}
	\bibliography{kms}

@article{Alon86,
    author = {N. Alon},
    title ={Eigenvalues and expanders} ,
    journal = {Combinatorica},
    year = {1986},
    volume={6},
    page={86-96},
}

@article{BarRou22,
    author = {I. Bardet and C. Rouz\'{e}},
    title = {Hypercontractivity and Logarithmic {S}obolev Inequality for Non-primitive Quantum Markov Semigroups and Estimation of Decoherence Rates},
    journal ={Ann. Henri Poincar\'{e}} ,
    year = {2022},
    volume={23}, 
    page={3839–3903}
}

@article{BDR20,
    author ={Beigi, S. and Datta, N. and Rouzé, C. } ,
    title = {Quantum Reverse Hypercontractivity: Its Tensorization and Application to Strong Converses.},
    journal = {Commun. Math. Phys. },
    year = {2020},
    volume={376},
    page={753–794},
}

@article{BGJ22,
title = {Complete logarithmic {S}obolev inequalities via {R}icci curvature bounded below},
journal = {Advances in Mathematics},
volume = {394},
pages = {108129},
year = {2022},
issn = {0001-8708},
author = {M. Brannan and L. Gao and M. Junge}
}

@article{BGJ23,
      title={Complete Logarithmic {S}obolev Inequalities via {R}icci Curvature Bounded Below {II}}, 
      author={M. Brannan and L. Gao and M. Junge},
      year={2023},
      journal = {Journal of Topology and Analysis},
      volume={15},
      page={741-794},
}

@article{CarMaa17,
author = {Eric A. Carlen and Jan Maas},
title = {{Gradient flow and entropy inequalities for quantum Markov semigroups with detailed balance}},
volume = {273},
journal = {Journal of Functional Analysis},
pages = {1810-1869},
year = {2017},
}

@article{CarMar15,
    author = {Carbone, R. and Martinelli, A.},
    title = {Logarithmic Sobolev inequalities in non-commutative algebras.},
    journal ={Infin. Dimens. Anal. Quantum Probab. Relat. Top.} ,
    year = {2015},
    volume={18},
    number={2},
    page={1550011}
}

@article{Cheeger70,
    author = {J. Cheeger},
    title = {A lower bound for the lowest eigenvalue of the Laplacian},
    journal ={Problems in Analysis: A Symposium in Honor of S. Bochner (R. C. Gunning, ed.), Princeton Univ. Press, Princeton, N.J.} ,
    year = {1970},
    page={195-199},
}

@article{GKS76,
    author ={V. Gorini and A. Kossakowski and E. C. G. Sudarshan} ,
    title = {Completely positive dynamical semigroups of $N$-level systems},
    volume={17},
    page={821},
    journal = {J. Math. Phys},
    year = {1976},
}

@article{FagUma10,
    author ={F. Fagnola and V. Umanit\'{a}} ,
    title ={Generators of {KMS} Symmetric {M}arkov Semigroups on {B(H)}: Symmetry and Quantum Detailed Balance} ,
    journal = {Commun. Math. Phys.},
    volume={ 298},
    pages={523--547},
    year = 2010,
}

@article{GaoRou22,
    author = {Gao, Li and Rouz\'{e}, Cambyse},
    title = {Complete Entropic Inequalities for Quantum {Markov} Chains},
    journal = {Archive for Rational Mechanics and Analysis},
    number={1},
    volume={245},
    year = {2022}
}

@article{GJLL25,
      title = {Complete positivity order and relative entropy decay},
      author = {L. Gao and M. Junge and N. LaRacuente and H. Li},
      year = {2025},
      journal = {Forum of Mathematics, Sigma},
      volume = {13},
      pages = {e31},
      doi = {10.1017/fms.2024.117},
}

@article{KasTem13,
    author = {M. J. Kastoryano and K. Temme},
    title = {Quantum logarithmic Sobolev inequalities and rapid mixing},
    journal ={Journal of Mathematical Physics} ,
    year = {2013},
    volume={54},
    number={5},
}

@Article{KFGV77,
  author       = {A. Kossakowski and A. Frigerio and V. Gorini and M. Verri},
  journal = {Commun. Math. Phys.},
  title        = {Quantum detailed balance and {KMS} conditions},
  doi          = { },
  number       = {},
  pages        = {97--110},
  url          = {},
  volume       = {57},
year         = {1977},
}

@article{Lin76,
  author    = {Göran Lindblad},
  title     = {{O}n the Generators of Quantum Dynamical Semigroups},
  journal   = {Commun. Math. Phys.},
  volume    = {48},
  pages     = {119--130},
  year      = {1976},
  doi       = {10.1007/BF01608499},
  url       = {https://doi.org/10.1007/BF01608499}
}

@article{LawSok88,
 author = {Gregory F. Lawler and Alan D. Sokal},
 journal = {Transactions of the American Mathematical Society},
 number = {2},
 pages = {557--580},
 publisher = {American Mathematical Society},
 title = {Bounds on the $L^2$ Spectrum for Markov Chains and Markov Processes: A Generalization of {C}heeger's Inequality},
 volume = {309},
 year = {1988}
}

@article{OlkZeg99,
    author ={Olkiewicz, R. and Zegarlinski, B.} ,
    title = {Hypercontractivity in noncommutative $L_p$ spaces},
    journal ={J. Funct. Anal.} ,
    year = {1999},
    volume={161}, 
    number={1},
    page={246-285}
}

@article{RicXu16,
author = {{\'E}ric Ricard and Quanhua Xu},
title = {{A noncommutative martingale convexity inequality}},
volume = {44},
journal = {The Annals of Probability},
number = {2},
publisher = {Institute of Mathematical Statistics},
pages = {867 -- 882},
year = {2016},
doi = {10.1214/14-AOP990},
URL = {https://doi.org/10.1214/14-AOP990}
}

@article{TKR10,
    author = {K. Temme and M. J. Kastoryano and M. B. Ruskai and M. M.
Wolf and F. Verstraete},
    title ={The $\chi^2$-Divergence and Mixing Times of Quantum Markov Processes} ,
    journal ={Journal of Mathematical Physics } ,
    year = {2010},
    page={122201},
    volume={51},
}

@article{TPK14,
    author ={Temme, K. and Pastawski, F. and Kastoryano, M.J.:},
    title ={Hypercontractivity of quasi-free quantum semigroups} ,
    journal = {Journal of Physics A: Mathematical and Theoretical},
    year = {2014},
    volume={47},
    page={405303}
}

@article{VerWir23,
    author ={M. Vernooij and M. Wirth} ,
    title ={Derivations and {KMS}-Symmetric Quantum {M}arkov Semigroups} ,
    journal ={Commun. Math. Phys.} ,
volume={403},
pages={381-416},
    year = {2023},
}

@article{Wir26,
    author = {M. Wirth},
    title = {The {KMS} and {GNS} Spectral Gap of Quantum {Markov} Semigroups},
    journal = {ArXiv:2604.21630v1},
    year = {2026}
}

@article{LWW26,
      author      = {Zhengwei Liu and Jincheng Wan and Jinsong Wu},
      journal     = {ArXiv:2609.13726},
      number      = {},
      title       = {Some entropic inequalities for primitive {KMS}-symmetric Quantum {Markov} Semigroups}, 
      volume      = {},
      year        = {2026},
}

@article{MFOSM2013,
       author     = {Müller-Lennert, Martin and Dupuis, Frédéric and Szehr, Oleg and Fehr, Serge and Tomamichel, Marco},
       journal    = {Journal of Mathematical Physics},
       number     = {12},
       title      = {On quantum Rényi entropies: A new generalization and some properties},
       volume     = {54},
       year       = {2013},
}
	
\end{document}